\pdfoutput=1
\RequirePackage{ifpdf}
\ifpdf 
\documentclass[pdftex]{sigma}
\else
\documentclass{sigma}
\fi

\newcommand{\ext}{{\rm ext}}

\newcommand{\End}{{\rm{End}\ts}}

\newcommand{\diag}{{\rm diag}}

\newcommand{\non}{\nonumber}
\newcommand{\wt}{\widetilde}
\newcommand{\wh}{\widehat}

\newcommand{\ot}{\otimes}
\newcommand{\la}{\lambda}

\newcommand{\al}{\alpha}

\newcommand{\ep}{\epsilon}
\newcommand{\ka}{\kappa}

\newcommand{\si}{\sigma}
\newcommand{\vs}{\varsigma}
\newcommand{\vp}{\varphi}
\newcommand{\de}{\delta}
\newcommand{\ze}{\zeta}

\newcommand{\hra}{\hookrightarrow}
\newcommand{\ve}{\varepsilon}
\newcommand{\ts}{\,}
\newcommand{\vac}{\mathbf{1}}

\newcommand{\qin}{q^{-1}}
\newcommand{\tss}{\hspace{1pt}}

\newcommand{\U}{ {\rm U}}

\newcommand{\CC}{\mathbb{C}\tss}

\newcommand{\ZZ}{\mathbb{Z}\tss}

\newcommand{\X}{ {\rm X}}
\newcommand{\Y}{ {\rm Y}}

\newcommand{\Sc}{\mathcal{S}}

\newcommand{\Cc}{\mathcal{C}}

\newcommand{\gl}{\mathfrak{gl}}

\newcommand{\oa}{\mathfrak{o}}
\newcommand{\spa}{\mathfrak{sp}}

\newcommand{\g}{\mathfrak{g}}

\newcommand{\sll}{\mathfrak{sl}}
\newcommand{\whg}{\wh{\g}}

\newcommand{\tra}{ {\rm t}}

\newcommand{\bi}{\bar{\imath}}
\newcommand{\bj}{\bar{\jmath}}

\newcommand{\Sym}{\mathfrak S}

\newcommand{\Fand}{\qquad\text{and}\qquad}
\newcommand{\For}{\qquad\text{or}\qquad}

\numberwithin{equation}{section}

\newtheorem{Theorem}{Theorem}[section]
\newtheorem{Lemma}[Theorem]{Lemma}
\newtheorem{Proposition}[Theorem]{Proposition}
\newtheorem{Corollary}[Theorem]{Corollary}

\theoremstyle{definition}
\newtheorem{Definition}[Theorem]{Definition}
\newtheorem{Remark}[Theorem]{Remark}

\begin{document}

\allowdisplaybreaks

\newcommand{\arXivNumber}{2008.07847}

\renewcommand{\thefootnote}{}

\renewcommand{\PaperNumber}{145}

\FirstPageHeading

\ShortArticleName{Representations of Quantum Affine Algebras in their $R$-Matrix Realization}

\ArticleName{Representations of Quantum Affine Algebras\\ in their $\boldsymbol{R}$-Matrix Realization\footnote{This paper is a~contribution to the Special Issue on Representation Theory and Integrable Systems in honor of Vitaly Tarasov on the 60th birthday and Alexander Varchenko on the 70th birthday. The full collection is available at \href{https://www.emis.de/journals/SIGMA/Tarasov-Varchenko.html}{https://www.emis.de/journals/SIGMA/Tarasov-Varchenko.html}}}

\Author{Naihuan JING~$^\dag$, Ming LIU~$^\ddag$ and Alexander MOLEV~$^\S$}

\AuthorNameForHeading{N.~Jing, M.~Liu and A.~Molev}

\Address{$^\dag$~Department of Mathematics, North Carolina State University, Raleigh, NC 27695, USA}
\EmailD{\href{mailto:jing@math.ncsu.edu}{jing@math.ncsu.edu}}

\Address{$^\ddag$~School of Mathematics, South China University of Technology, Guangzhou, 510640, China}
\EmailD{\href{mailto:mamliu@scut.edu.cn}{mamliu@scut.edu.cn}}

\Address{$^\S$~School of Mathematics and Statistics, University of Sydney, NSW 2006, Australia}
\EmailD{\href{mailto:alexander.molev@sydney.edu.au}{alexander.molev@sydney.edu.au}}

\ArticleDates{Received August 19, 2020, in final form December 25, 2020; Published online December 28, 2020}

\Abstract{We use the isomorphisms between the $R$-matrix and Drinfeld presentations of the quantum affine algebras in types $B$, $C$ and $D$ produced in our previous work to describe finite-dimensional irreducible representations in the $R$-matrix realization.
We also review the isomorphisms for the Yangians of these types and use Gauss decomposition to establish an equivalence of the descriptions of the representations in the $R$-matrix and Drinfeld presentations of the Yangians.}


\Classification{17B37}

\renewcommand{\thefootnote}{\arabic{footnote}}
\setcounter{footnote}{0}

\section{Introduction}\label{sec:int}

The Yangians and quantum affine algebras associated with simple
Lie algebras comprise two remarkable families of
infinite-dimensional quantum groups, as introduced by V.~Drinfeld~\cite{d:ha}
and M.~Jimbo~\cite{j:qd}. Both families have since found numerous connections with
many areas in mathematics and physics, they
possess rich and versatile representation theory.

Finite-dimensional irreducible representations of the Yangians were classified
by Drinfeld in his paper \cite{d:nr} by relying on the pioneering
work by V.~Tarasov~\cite{t:sq, t:im}. The classification in~\cite{d:nr} uses
a new presentation of the Yangians
and quantum affine algebras which is now often referred to
as the {\em Drinfeld presentation} and which was used by V.~Chari and A.~Pressley
to classify finite-dimensional irreducible representations
of quantum affine algebras \cite[Chapter~12]{cp:gq}.

A different kind of presentations of these algebras known as {\em $R$-matrix presentations}
goes back to the work of the Leningrad school of L.~Faddeev; see, e.g.,
\cite{fri:qa, ks:qs, rs:ce,rtf:ql}.
In accordance to Drinfeld~\cite{d:qg},
such presentations can be produced from
the universal $R$-matrix associated with a quantum group
and they can be associated with arbitrary
finite-dimensional representations.
This approach was developed further in
a more recent work~\cite{w:rm}.

Explicit isomorphisms between the $R$-matrix and Drinfeld
presentations of the Yangians in type $A$ were given in the original work
\cite{d:nr}, while detailed proofs were produced
by J.~Brundan and A.~Kleshchev~\cite{bk:pp}. An analogous isomorphism
for the type $A$ quantum affine algebras is due to
J.~Ding and I.~Frenkel~\cite{df:it}.

Despite their importance in representation theory and applications,
such isomorphisms had remained unknown
beyond type $A$ until recent work \cite{grw:eb,jlm:ib,jlm:ib-c,jlm:ib-bd}, where they
were produced for the remaining classical types~$B$,~$C$ and~$D$.

Our goal in this paper is to give a brief review
of Yangians and quantum affine algebras in types $B$, $C$ and $D$
and apply the isomorphisms to describe
finite-dimensional irreducible representations of these algebras
in their $R$-matrix realization.

In the case of Yangians, such a description was already given in~\cite{amr:rp},
so that the isomorphisms connect two sides of the representation theory
and explain how additional symmetries of the representation parameters
arise from the Gauss decomposition of the generator matrices.
However, for the $R$-matrix realization of the quantum affine algebras the isomorphisms are essential
to get a parametrization of their finite-dimensional irreducible representations.

A key step in our arguments relies on a consistency property for
the triangular decomposition of the algebra and
the Gauss decomposition of the generator matrices.
The property ensures that the upper triangular Gaussian generators can be expressed
in terms of the `simple root generators' implying that the annihilation properties of the highest vectors
agree in both presentations of the algebras.
Its Yangian version is well-known
in type~$A$ (see, e.g., \cite[Section~3.1]{m:yc}), while proofs in the remaining
classical types follow from \cite[Lemma~5.15]{jlm:ib}.
This property has also been known for the quantum affine algebras of type $A$ \cite{kpt:ob}, so that
the isomorphism of \cite{df:it} connects the descriptions of the representations
given in \cite[Section~12.2]{cp:gq} and~\cite{gm:rt}. Similar to the approach of~\cite{kpt:ob},
we establish the consistency property
in Section~\ref{subsubsec:td} in types $B$, $C$ and $D$ by deriving it from the defining relations.
An alternative way
can rely on explicit formulas for the universal $R$-matrices
which, however, leads to more
involved calculations.

In Appendix~\ref{asec:isy} we give a modified version of the Yangian isomorphisms
produced in~\cite{jlm:ib}, which is based on the opposite
Gauss decomposition of the generator matrix. This version
can be applied to make an alternative connection
between the parameters of representations in the two realizations
of the Yangians.

\section[Representations of Yangians of types B, C and D]{Representations of Yangians of types $\boldsymbol{B}$, $\boldsymbol{C}$ and $\boldsymbol{D}$}\label{sec:ry}

\subsection{Classification results}\label{subsec:cry}

We will denote by $\g$ one of the simple Lie algebras of type $B_n$, $_n$ or $D_n$. That is,
$\g=\oa_N$ is either the orthogonal Lie algebra with $N=2n+1$ and $N=2n$,
or $\g=\spa_{N}$ is the symplectic Lie algebra with $N=2n$.
Choose simple roots in the form
\begin{gather}\label{aroots}
\al_i=\ep_i-\ep_{i+1}\qquad\text{for}\quad i=1,\dots,n-1,
\end{gather}
and
\begin{gather*}
\al_n=\begin{cases}\ep_n &\text{if}\ \g=\oa_{2n+1},\\
2\tss\ep_n &\text{if}\ \g=\spa_{2n},\\
\ep_{n-1}+\ep_n &\text{if}\ \g=\oa_{2n},
\end{cases}
\end{gather*}
where
$\ep_1,\dots,\ep_n$ is an orthonormal basis of a Euclidian space with the inner product
$(\cdot\ts,\cdot)$. The Cartan matrix $[A_{ij}]$ for $\g$ is given by
\begin{gather}\label{cartan}
A_{ij}=\frac{2(\al_i,\al_j)}{(\al_i,\al_i)}.
\end{gather}
We will use the notation
\begin{gather}\label{ri}
r_i=\tfrac12\tss(\al_i,\al_i),\qquad i=1,\dots,n.
\end{gather}

As defined in~\cite{d:nr},
the {\em Drinfeld Yangian} $\Y^D(\g)$ is generated by
elements $\kappa^{}_{i\tss r}$, $\xi_{i\tss r}^{+}$ and
$\xi_{i\tss r}^{-}$ with $i=1,\dots,n$ and $r=0,1,\dots$
subject to the defining relations
\begin{gather*}
[\kappa^{}_{i\tss r},\kappa^{}_{j\tss s}]=0,\\
\big[\xi_{i\tss r}^{+},\xi_{j\tss s}^{-}\big]=\de_{ij}\ts\kappa^{}_{i\ts r+s},\\
\big[\kappa^{}_{i\tss 0},\xi_{j\tss s}^{\pm}\big]=\pm\ts (\al_i,\al_j)\ts\xi_{j\tss s}^{\pm},\\
\big[\kappa^{}_{i\ts r+1},\xi_{j\tss s}^{\pm}\big]-\big[\kappa^{}_{i\tss r},\xi_{j\ts s+1}^{\pm}\big]=
\pm\ts\frac{(\al_i,\al_j)}{2}\big(\kappa^{}_{i\tss r}\ts \xi_{j\tss s}^{\pm}
+\xi_{j\tss s}^{\pm}\ts \kappa^{}_{i\tss r}\big),\\
\big[\xi_{i\ts r+1}^{\pm},\xi_{j\tss s}^{\pm}\big]-\big[\xi_{i\tss r}^{\pm},\xi_{j\ts s+1}^{\pm}\big]=
\pm\ts\frac{(\al_i,\al_j)}{2}\big(\xi_{i\tss r}^{\pm}\ts \xi_{j\tss s}^{\pm}
+\xi_{j\tss s}^{\pm}\ts \xi_{i\tss r}^{\pm}\big),\\
\sum_{p\in\Sym_m} \big[\xi_{i\tss r_{p(1)}}^{\pm},\big[\xi_{i\tss r_{p(2)}}^{\pm},\dots
\big[\xi_{i\tss r_{p(m)}}^{\pm},\xi_{j\tss s}^{\pm}\big]\dots \big]\big]=0,
\end{gather*}
where the last relation holds for all $i\ne j$, and we denoted $m=1-a_{ij}$.

By the results of \cite{d:nr} (they apply to any simple Lie algebra~$\g$),
every finite-dimensional irreducible representation
of the algebra $\Y^D(\g)$ is generated by a nonzero vector $\ze$ (called the {\em highest vector})
which is annihilated by all $\xi_{i\tss r}^{+}$ and is a simultaneous eigenvector for all
$\kappa^{}_{i\tss r}$ so that
\begin{gather*}
\kappa^{}_{i\tss r}\ze=d_{i\tss r}\tss \ze,\qquad d_{i\tss r}\in\CC.
\end{gather*}
Furthermore, there exist unique monic polynomials $P_1(u),\dots,P_n(u)$ in $u$
such that
\begin{gather*}
1+\sum_{r=0}^{\infty}d_{i\tss r}\tss u^{-r-1}=\frac{P_i(u+r_i)}{P_i(u)},\qquad i=1,\dots,n.
\end{gather*}
Equivalently, every finite-dimensional irreducible representation
of the algebra $\Y^D(\g)$ is generated by a nonzero vector $\ze'$ (the {\em highest vector}
with respect to the opposite triangular decomposition)
which is annihilated by all $\xi_{i\tss r}^{-}$ and is a simultaneous eigenvector for all
$\kappa^{}_{i\tss r}$ so that
\begin{gather*}
\kappa^{}_{i\tss r}\ze'=d^{\ts\prime}_{i\tss r}\tss \ze',\qquad d^{\ts\prime}_{i\tss r}\in\CC.
\end{gather*}
Furthermore, there exist unique monic polynomials $Q_1(u),\dots,Q_n(u)$ in $u$
such that
\begin{gather}\label{qpo}
1+\sum_{r=0}^{\infty}d^{\ts\prime}_{i\tss r}\tss u^{-r-1}=\frac{Q_i(u)}{Q_i(u+r_i)},\qquad i=1,\dots,n.
\end{gather}
All possible $n$-tuples of monic polynomials $(P_1(u),\dots,P_{n}(u))$
and $(Q_1(u),\dots,Q_{n}(u))$
arise in this way.
The equivalence of the two parametrizations is seen by using the automorphism of
the algebra $\Y^D(\g)$ defined by
\begin{gather*}
\xi_{i\tss r}^{+}\mapsto (-1)^{r+1}\tss\xi_{i\tss r}^{-},\qquad
\xi_{i\tss r}^{-}\mapsto (-1)^{r+1}\tss\xi_{i\tss r}^{+},\qquad
\ka_{i\tss r}\mapsto (-1)^{r+1}\tss\ka_{i\tss r}.
\end{gather*}

\subsection{Gaussian generators and isomorphisms}\label{subsec:rmy}

Introduce the following elements of the endomorphism algebra $\End\CC^{N}\ot\End\CC^{N}$:
\begin{gather*}
P=\sum_{i,j=1}^{N}e_{ij}\otimes e_{ji}\Fand
Q=\sum_{i,j=1}^N \ve_i\tss\ve_j\ts e_{ij}
\ot e_{i'j'},
\end{gather*}
where $e_{ij}\in\End\CC^{N}$ are the matrix units, and we use the notation
$i'=N+1-i$ and
set $\ve_i\equiv 1$ in the orthogonal case, and
\begin{gather*}
\ve_i=\begin{cases} \phantom{-}1& \text{for} \ i=1,\dots,n,\\
-1 & \text{for} \ i=n+1,\dots,2n,
\end{cases}
\end{gather*}
in the symplectic case. Set
\begin{gather*}
\kappa=\begin{cases} N/2-1& \text{in the orthogonal case,}\\
N/2+1& \text{in the symplectic case}.
\end{cases}
\end{gather*}
Following \cite{zz:rf}, consider
the $R$-{\em matrix} $R(u)$
\begin{gather*}
R(u)=1-\frac{P}{u}+\frac{Q}{u-\kappa}.
\end{gather*}

The {\em extended Yangian} $\X(\g)$ is defined as a unital associative algebra with generators
$t_{ij}^{(r)}$, where $1\leqslant i,j\leqslant N$ and $r=1,2,\dots$, satisfying certain quadratic relations.
Introduce the formal series
\begin{gather*}
t_{ij}(u)=\de_{ij}+\sum_{r=1}^{\infty}t_{ij}^{(r)}\ts u^{-r}
\in\X(\g)\big[\big[u^{-1}\big]\big]
\end{gather*}
and set
\begin{gather*}
T(u)=\sum_{i,j=1}^N e_{ij}\ot t_{ij}(u)
\in \End\CC^N\ot \X(\g)\big[\big[u^{-1}\big]\big].
\end{gather*}
Consider the tensor product algebra $\End\CC^{N}\ot\End\CC^{N}\ot \X(\g)$
and introduce the series with coefficients in this algebra by
\begin{gather*}
T_1(u)=\sum\limits_{i,j=1}^{N} e_{ij}\ot 1\ot t_{ij}(u)
\Fand
T_2(u)=\sum\limits_{i,j=1}^{N}1 \ot e_{ij}\ot t_{ij}(u).
\end{gather*}
The defining relations for the algebra $\X(\g)$ are then written in the form
\begin{gather*}
R(u-v)\ts T_1(u)\ts T_2(v)=T_2(v)\ts T_1(u)\ts R(u-v).
\end{gather*}
The {\em Yangian}
$\Y(\g)$ is defined as the subalgebra of
$\X(\g)$ which
consists of the elements stable under
the automorphisms
\begin{gather*}
\mu_f\colon \ T(u)\mapsto f(u)\ts T(u),
\end{gather*}
for all series $f(u)=1+f_1u^{-1}+f_2 u^{-2}+\cdots$ with $f_i\in\CC$.
Equivalently, $\Y(\g)$ is isomorphic to the quotient of the algebra $\X(\g)$ by the relation
\begin{gather}\label{unitaryint}
T^{\tss\tra}(u+\ka)\ts T(u)=1,
\end{gather}
where $\tra$ denotes the matrix
transposition with $e_{ij}^{\tra}=\ve_i\ts\ve_j\ts e_{j',i'}$.

\subsubsection{Isomorphisms}\label{subsubsec:isy}

Apply the Gauss decomposition to the matrix $T(u)$,
\begin{gather}\label{gd}
T(u)=F(u)\ts H(u)\ts E(u),
\end{gather}
where $F(u)$, $H(u)$ and $E(u)$ are uniquely determined matrices of the form
\begin{gather*}
F(u)=\begin{bmatrix}
1&0&\dots&0\ts\\
f_{21}(u)&1&\dots&0\\
\vdots&\vdots&\ddots&\vdots\\
f_{N1}(u)&f_{N2}(u)&\dots&1
\end{bmatrix},
\qquad
E(u)=\begin{bmatrix}
\ts1&e_{12}(u)&\dots&e_{1N}(u)\ts\\
\ts0&1&\dots&e_{2N}(u)\\
\vdots&\vdots&\ddots&\vdots\\
0&0&\dots&1
\end{bmatrix},
\end{gather*}
and $H(u)=\diag\ts [h_1(u),\dots,h_N(u) ]$. Define the series
with coefficients in $\Y(\g)$ by
\begin{gather*}
\kappa^{}_i(u)=h_i (u-(i-1)/2 )^{-1}\ts h_{i+1} (u-(i-1)/2 )
\end{gather*}
for $i=1,\dots,n-1$, and
\begin{gather*}
\kappa^{}_n(u)=\begin{cases}
h_n (u-(n-1)/2 )^{-1}\ts h_{n+1} (u-(n-1)/2 )
 &\text{for}\ \oa_{2n+1},\\
\tss h_n (u-n/2 )^{-1}\ts h_{n+1} (u-n/2 )
 &\text{for}\ \spa_{2n},\\
h_{n-1} (u-(n-2)/2 )^{-1}\ts h_{n+1} (u-(n-2)/2 )
 &\text{for}\ \oa_{2n}.
\end{cases}
\end{gather*}
Furthermore, set
\begin{gather*}
\xi_i^+(u)=f_{i+1\ts i} (u-(i-1)/2 ),\qquad
\xi_i^-(u)=e_{i\ts i+1} (u-(i-1)/2 )
\end{gather*}
for $i=1,\dots,n-1$,
\begin{gather*}
\xi_n^+(u)=\begin{cases}
f_{n+1\ts n}(u-(n-1)/2)
 &\text{for}\ \oa_{2n+1},\\
f_{n+1\ts n} (u-n/2 )
 &\text{for}\ \spa_{2n},\\
f_{n+1\ts n-1} (u-(n-2)/2 )
 &\text{for}\ \oa_{2n}
\end{cases}
\end{gather*}
and
\begin{gather*}
\xi_n^-(u)=\begin{cases}
e_{n\ts n+1} (u-(n-1)/2 )
&\text{for}\ \oa_{2n+1},\\
\frac{1}{2}e_{n\ts n+1} (u-n/2 )
&\text{for}\ \spa_{2n},\\
e_{n-1\ts n+1} (u-(n-2)/2 )
&\text{for}\ \oa_{2n}.
\end{cases}
\end{gather*}
Introduce elements of $\Y(\g)$ by the respective expansions into power series in $u^{-1}$,
\begin{gather*}
\kappa^{}_i(u)=1+\sum_{r=0}^{\infty}\kappa^{}_{i\tss r}\ts u^{-r-1}\Fand
\xi_i^{\pm}(u)=\sum_{r=0}^{\infty}\xi_{i\tss r}^{\pm}\ts u^{-r-1}
\end{gather*}
for $i=1,\dots,n$. According to \cite[Main Theorem]{jlm:ib},
the mapping which sends the genera\-tors~$\kappa^{}_{i\tss r}$ and $\xi_{i\tss r}^{\pm}$ of $\Y^{D}(\g)$ to the elements
of $\Y(\g)$ with the same names defines an isomorphism $\Y^{D}(\g)\cong \Y(\g)$.

\subsubsection{Central elements of the extended Yangian}\label{subsev:zuy}

A presentation of the extended Yangian $\X(\g)$ in terms
of the Gaussian generators is given in
\cite[Theorem~5.14]{jlm:ib}.\footnote{Note that the formulas in \cite[equations~(5.4) and (5.47)]{jlm:ib}
should be corrected by swapping the order of the factors on their right hand sides.}
By \cite[Theorem~5.8]{jlm:ib},
all coefficients of the series
\begin{gather*}
z(u)=\begin{cases}\displaystyle\prod_{i=1}^n h_i(u+\ka-i)^{-1}\tss \prod_{i=1}^n h_i(u+\ka-i+1)\cdot
h_{n+1}(u)\tss h_{n+1}(u-1/2)
&\text{if}\ N=2n+1,\vspace{1mm}\\
\displaystyle\prod_{i=1}^{n-1} h_i(u+\ka-i)^{-1}\tss \prod_{i=1}^n h_i(u+\ka-i+1)\cdot
h_{n+1}(u)
&\text{if}\ N=2n,
\end{cases}
\end{gather*}
belong to the center $\Cc$ of the extended Yangian $\X(\g)$. Moreover,
these coefficients generate the center and we have the tensor product decomposition
\begin{gather*}
\X(\g)\cong \Y(\g)\ot\Cc.
\end{gather*}
The following identity holds in $\X(\g)$:
\begin{gather*}
T^{\tss\tra}(u+\ka)\ts T(u)=z(u),
\end{gather*}
so that the Yangian $\Y(\g)$ is isomorphic to the quotient of $\X(\g)$ by the relation $z(u)=1$. We will record the relations which follow from the arguments in \cite[Section~5.3]{jlm:ib}.

\begin{Lemma}\label{lem:ygaurell}
In the algebra $\X(\g)$ we have
\begin{gather}
h_1(u+\ka)\tss h_{1'}(u)=z(u),\nonumber \\ 
h_i(u+\ka-i)\tss h_{i'}(u)=h_{i+1}(u+\ka-i)\tss h_{(i+1)'}(u),\label{yhiipone}
\end{gather}
where $i=1,\dots,n$ for $N=2n+1$, and $i=1,\dots,n-1$ for $N=2n$.
\end{Lemma}

\subsection{Highest weight representations}\label{subsec:hwy}

\begin{Definition}\label{def:hwry}
A representation $V$ of the algebra $\Y(\g)$ (or $\X(\g)$)
is called a {\it highest weight representation} if there
exists a nonzero vector $\ze\in V$ such that $V$ is generated by $\ze$
and the following relations hold:
\begin{alignat*}{2}
&t_{ij}(u)\ts\ze=0 \qquad &&\text{for} \quad
1\leqslant i<j\leqslant N, \qquad \text{and}\\
&t_{ii}(u)\ts\ze=\la_i(u)\ts\ze \qquad &&\text{for} \quad
i=1,\dots,N,
\end{alignat*}
for some formal series $\la_i(u)\in 1+u^{-1}\tss\CC\big[\big[u^{-1}\big]\big]$.
The vector $\ze$ is called the {\it highest vector} of the representation~$V$.
\end{Definition}

The following classification theorem for finite-dimensional irreducible representations
of the Yangians in types $B$, $C$ and $D$ was proved in~\cite{amr:rp}
in terms of their $R$-matrix presentation. We will use
the isomorphisms of~\cite{jlm:ib} which we recalled in Section~\ref{subsubsec:isy},
to make an explicit connection between this theorem and the results of~\cite{d:nr}.
Note that such a connection
was already established in~\cite{grw:eb}, where isomorphisms between
three presentations of the orthogonal and symplectic Yangians were constructed.
However, those results did not use the Gaussian presentation
which we will rely on in our arguments.

\begin{Theorem}\label{thm:yclassi}\quad\samepage
\begin{enumerate}
\item[$1.$]
Any finite-dimensional irreducible representation of the algebra $\Y(\g)$
is a highest weight representation. Its parameters satisfy the relations
\begin{gather*}
\frac{\la_i(u)}{\la_{i+1}(u)}=\frac{P_i(u+1)}{P_i(u)},
\qquad i=1,\dots,n-1,
\end{gather*}
and
\begin{alignat*}{2}
\frac{\la_n(u)}{\la_{n+1}(u)}&=\frac{P_n(u+1/2)}{P_n(u)}\qquad&&\text{for type $B_n$},\\
\frac{\la_n(u)}{\la_{n+1}(u)}&=\frac{P_n(u+2)}{P_n(u)}\qquad&&\text{for type $C_n$},\\
\frac{\la_{n-1}(u)}{\la_{n+1}(u)}&=\frac{P_n(u+1)}{P_n(u)}\qquad&&\text{for type $D_n$},
\end{alignat*}
for some monic polynomials $P_i(u)$ in $u$.

\item[$2.$] Every $n$-tuple $(P_1(u),\dots,P_n(u))$ of monic polynomials in $u$ arises in this way.

\item[$3.$] The series $\la_i(u)$ satisfy the relations
\begin{gather}\label{ylarel}
\la_i(u+\ka-i)\tss \la_{i'}(u)=\la_{i+1}(u+\ka-i)\tss \la_{(i+1)'}(u),
\end{gather}
where $i=0,1,\dots,n$ for $N=2n+1$, and $i=0,1,\dots,n-1$ for $N=2n$, and we set $\la_0(u)=\la_{0'}(u):=1$.
\end{enumerate}
\end{Theorem}

\begin{proof}Using the isomorphism $\Y^{D}(\g)\cong \Y(\g)$ and the classification results
recalled in Section~\ref{subsec:cry}, we find that
any finite-dimensional irreducible representation $V$ of the algebra $\Y(\g)$
in types $B_n$ and $C_n$
is generated by a vector $\ze'$ such that
\begin{alignat}{2}
&e_{i,i+1}(u)\tss\ze'=0\qquad&&\text{for}\quad i=1,\dots,n,\label{yeiipoze}\\
& h_{i}(u)\ts\ze'=\la_i(u)\ts\ze'\qquad&&\text{for}\quad i=1,\dots,n+1,\label{yhze}
\end{alignat}
for some formal series $\la_i(u)\in1+u^{-1}\CC\big[\big[u^{-1}\big]\big]$. For type $D_n$,
the same conditions hold, except that relation \eqref{yeiipoze} with $i=n$
should be replaced with $e_{n-1,n+1}(u)\tss\ze'=0$. Indeed, for all types,
relation~\eqref{yeiipoze}
is clear from the definition of the highest vector $\ze'$, while
\eqref{yhze} follows from the condition that $\ze'$ is an eigenvector
for all series $\kappa^{}_i(u)$ with $i=1,\dots,n$ and $z(u)\tss \ze'=\ze'$.
Now, Lemma~\ref{lem:ygaurell} and \cite[Lemma~5.15]{jlm:ib} imply that
\begin{alignat}{2}
e_{ij}(u)\tss\ze'&=0\qquad&&\text{for}\quad i<j
\non\\
h_{i}(u)\ts\ze'&=\la_i(u)\ts\ze'\qquad&&\text{for}\quad i=1,\dots,N,
\non
\end{alignat}
for certain formal series $\la_i(u)\in1+u^{-1}\CC\big[\big[u^{-1}\big]\big]$ satisfying
identities \eqref{ylarel}. Finally, note that the values of the parameters~\eqref{ri}
are found by $r_i=1$ for $i=1,\dots,n-1$, while $r_n=1/2$ for type $B_n$, $r_n=2$ for type $C_n$, and $r_n=1$ for type $D_n$, so that
conditions in Part~1 of the theorem follow from~\eqref{qpo}.
\end{proof}

\begin{Corollary}\label{cor:yrepe}
All statements of Theorem~{\rm \ref{thm:yclassi}} hold in the same form for the extended Yangian~$\X(\g)$, except that the value $i=0$ is excluded for the conditions~\eqref{ylarel}.
\end{Corollary}

\begin{proof} The proof of the theorem obviously extends to the algebra $\X(\g)$. Relation~\eqref{ylarel} with $i=0$ is now replaced by the property that the series $z(u)$ acts in the highest weight representation as multiplication by $\la_{1}(u+\ka)\tss \la_{1'}(u)$.
\end{proof}

\begin{Remark}\label{rem:indpr}
It is clear that the arguments used in the proof of the theorem can be reversed, so that the classification theorem for the Yangian representations in types~$B$, $C$ and $D$ proved in~\cite{amr:rp} implies the corresponding results of~\cite{d:nr}.
\end{Remark}

\section{Representations of quantum affine algebras}\label{sec:rqa}

\subsection{Classification results}\label{subsec:cr}

We will suppose that $q$ is a nonzero complex number which is not a root of unity
and set $q_i=q^{r_i}$ for $i=1,\dots,n$, with $r_i$ defined in~\eqref{ri}.
The Cartan matrix of the simple Lie algebra~$\g$ of type~$B_n$,~$C_n$ or~$D_n$
is given by~\eqref{cartan}. We will use the standard notation
\begin{gather*}
[k]_q=\frac{q^k-q^{-k}}{q-q^{-1}}
\end{gather*}
for a nonnegative integer $k$, together with
\begin{gather*}
[k]_q!=\prod_{s=1}^{k}[s]_q\Fand
\begin{bmatrix}k\\r\end{bmatrix}_{q}=\frac{[k]_q!}{[r]_q!\ts[k-r]_q!}.
\end{gather*}

The {\em quantum affine algebra} $\U_q(\whg)$ (with the trivial central charge)
in its Drinfeld presentation
is the associative algebra
with generators
$x_{i,m}^{\pm}$, $a_{i,l}$ and $k_{i}^{\pm1}$ for $i=1,\dots,n$ and
$m,l\in\ZZ$ with $l\ne 0$, subject to the following defining relations:
\begin{gather*}
 k_ik_j=k_jk_i, \qquad k_i\ts a_{j,l}=a_{j,l}\ts k_i,\qquad a_{i,m}\tss a_{j,l}=a_{j,l}\tss a_{i,m},
\\
k_i\ts x_{j,m}^{\pm}\ts k_i^{-1}=q_i^{\pm A_{ij}}x_{j,m}^{\pm},\qquad
 [a_{i,m}, x_{j,l}^{\pm}]=\pm \frac{[mA_{ij}]_{q_i}}{m}\ts x^{\pm}_{j,m+l},
\\
 x^{\pm}_{i,m+1}x^{\pm}_{j,l}-q_i^{\pm A_{ij}}x^{\pm}_{j,l}x^{\pm}_{i,m+1}=
 q_i^{\pm A_{ij}}x^{\pm}_{i,m}x^{\pm}_{j,l+1}-x^{\pm}_{j,l+1}x^{\pm}_{i,m},\\
 \big[x_{i,m}^{+},x_{j,l}^{-}\big] =\delta_{ij}\ts\frac{\psi_{i,m+l}
 -\varphi_{i,m+l}}{q_i-q_i^{-1}},
\\
 \sum_{\pi\in \Sym_{r}}\sum_{l=0}^{r}(-1)^l\begin{bmatrix}r\\l\end{bmatrix}_{q_i}
 x^{\pm}_{i,s_{\pi(1)}}\cdots x^{\pm}_{i,s_{\pi(l)}}
 x^{\pm}_{j,m}x^{\pm}_{i,s_{\pi(l+1)}}\cdots x^{\pm}_{i,s_{\pi(r)}}=0, \qquad i\neq j,
\end{gather*}
where in the last relation we set $r=1-A_{ij}$. The elements
$\psi_{i,m}$ and $\varphi_{i,-m}$ with $m\in \ZZ_+$ are defined by
\begin{gather*}
\psi_i(u):=\sum_{m=0}^{\infty}\psi_{i,m}u^{-m}=k_i\ts\exp\left(\big(q_i-q_i^{-1}\big)
\sum_{s=1}^{\infty}a_{i,s}u^{-s}\right),\\
\varphi_{i}(u):=\sum_{m=0}^{\infty}\varphi_{i,-m}u^{m}=k_i^{-1}\exp\left({-}\big(q_i-q_i^{-1}\big)
\sum_{s=1}^{\infty}a_{i,-s}u^{s}\right),
\end{gather*}
whereas $\psi_{i,m}=\varphi_{i,-m}=0$ for $m<0$.

\begin{Remark}\label{rem:ext}
For the Lie algebras $\g$ of types $C_n$ and $D_n$,
it will be convenient to work with an extended quantum algebra obtained by adjoining
the square roots $k_n^{\pm 1/2}$ and
$(k_{n-1}k_n)^{\pm 1/2}$, respectively. Accordingly, we need to add the defining relations
\begin{gather*}
k_n^{1/2}\tss x_{j,m}^{\pm}\tss k_n^{-1/2}=q^{\pm A_{nj}}x_{j,m}^{\pm}
\end{gather*}
for type $C_n$, and
\begin{gather*}
(k_{n-1}k_n)^{1/2}\tss x_{j,m}^{\pm}\tss (k_{n-1}k_n)^{-1/2}=q^{\pm (A_{n-1,j}+A_{nj})/2}x_{j,m}^{\pm}
\end{gather*}
for type $D_n$, while
the new elements commute with all the remaining
generators.
\end{Remark}

As explained in \cite[Section~12.2]{cp:gq}, the algebra $\U_q(\whg)$ admits a family of
sign automorphisms such that the composition of
any finite-dimensional irreducible representation with a suitable automorphism of this kind is
isomorphic to a {\em type $\vac$ representation}.
Such a representation is generated by a vector~$\ze$ which is
annihilated by all $x_{i,m}^{+}$ and is a simultaneous eigenvector for all~$k_i$
and $a_{i,l}$. Furthermore, if the series
$\Phi_i(u)\in\CC[[u]]$ and $\Psi_i(u)\in\CC\big[\big[u^{-1}\big]\big]$
are defined by
\begin{gather*}
\varphi_{i}(u)\tss \ze=\Phi_i(u)\tss \ze\Fand\psi_i(u)\tss \ze=\Psi_i(u)\tss \ze,
\end{gather*}
then there exist unique polynomials $P_1(u),\dots,P_n(u)$ in $u$ all with constant term $1$
such that
\begin{gather*}
\Phi_i(u)=q_i^{{-}\deg P_i}\ts \frac{P_i\big(u\tss q_i^{2}\big)}{P_i(u)}=
\Psi_i(u),\qquad i=1,\dots,n,
\end{gather*}
where the equalities are understood for the expansions of the rational functions in $u$
as series in~$u$ and~$u^{-1}$, respectively.\footnote{The roles of $u$ and $u^{-1}$
are swapped in our notation as compared to \cite[Section~12.2]{cp:gq}.}
Every $n$-tuple of polynomials $(P_1(u),\dots,P_{n}(u))$ in~$u$,
where each~$P_i(u)$ has constant term~$1$, arises in this way.

The above classification of finite-dimensional irreducible representations of
$\U_q(\whg)$ is valid for any simple Lie algebra~$\g$.
Note also that the corresponding results of \cite[Section~12.2]{cp:gq} apply
to centrally extended algebras $\U_q(\whg)$, which show that the central element
acts trivially in any finite-dimensional representation.
For this reason we only consider the quotients of the quantum affine algebras
by the relation specifying the value of the central element as equal to~$1$.

\subsection{Gaussian generators and isomorphisms}\label{subsec:rm}

To define $R$-matrix realizations of the quantum affine
algebras,
introduce elements of the endomorphism algebra $\End\CC^{N}\ot\End\CC^{N}$ by
\begin{gather*}
P=\sum_{i,j=1}^{N}e_{ij}\otimes e_{ji}\Fand
Q=\sum_{i,j=1}^{N} q^{\bi-\bj}\ts\ve_i\tss\ve_j\ts e_{i'j'}\otimes e_{ij},
\end{gather*}
where
$i'=N+1-i$, as before, and
\begin{gather*}
\big(\ts\overline{1},\overline{2},\dots,\overline{N}\ts\big)
=\begin{cases}
\big(n-\frac{1}{2},\dots,\frac{3}{2},\frac{1}{2},0,-\frac{1}{2},-\frac{3}{2},\dots,-n+\frac{1}{2}\big)
&\quad\text{for}\quad \g=\oa_{2n+1},\\
(n,\dots,2,1,-1,-2,\dots,-n)
&\quad\text{for}\quad \g=\spa_{2n},\\
(n-1,\dots,1,0,0,-1,\dots,-n+1)
&\quad\text{for}\quad \g=\oa_{2n}.
\end{cases}
\end{gather*}
Furthermore, introduce the $R$-matrix by
\begin{gather*}
R =q\ts\sum_{i=1, i\neq i'}^{N}e_{ii}\otimes e_{ii}+e_{n+1,n+1}\ot e_{n+1,n+1}
+\sum_{i\neq j,j'}e_{ii}\otimes e_{jj}
+q^{-1}\sum_{i\neq i'}e_{ii}\otimes e_{i'i'}\\
\hphantom{R=}{}+\big(q-q^{-1}\big)\sum_{i<j}e_{ij}\otimes
e_{ji}-\big(q-q^{-1}\big)\sum_{i>j}q^{\bi-\bj}\ts\ve_i\tss\ve_j\ts e_{i'j'}\otimes e_{ij},
\end{gather*}
where the second term $e_{n+1,n+1}\ot e_{n+1,n+1}$ should be omitted
if $N=2n$. This formula goes back to \cite{b:ts,b:iq,j:qr,rtf:ql}, which appeared along with the spectral
parameter-dependent $R$-matrix
\begin{gather}\label{ru}
R(u)=(u-1)R+\big(q-q^{-1}\big)\left(P-\frac{u\tss\xi-\xi}{u-\xi}\ts Q\right),
\end{gather}
where
\begin{gather*}
\xi=q^{-2\ka}=\begin{cases}q^{-N+2} &\text{if}\ \g=\oa_{N},\\
q^{-N-2} &\text{if}\ \g=\spa_{2n}.
\end{cases}
\end{gather*}
The corresponding two-parameter $R$-matrix $R(u,v)=v\tss R(u/v)$ can be written as
\begin{gather}
R(u,v)=\sum_{i,j=1}^N\tss \big(u\tss q^{\de_{ij}}-v\tss q^{-\de_{ij}}\big)\tss e_{ii}\ot e_{jj}
 +\big(q-\qin\big)\tss\sum_{i,j=1}^N(u\ts\delta_{i<j} +v\ts \delta_{i>j})\ts e_{ij}\ot e_{ji}\nonumber\\
\hphantom{R(u,v)=}{} -\ts\frac{u-v}{u-v\tss\xi}\ts
\sum_{i,j=1}^{N} d_{ij}(u,v)\ts e_{i'j'}\otimes e_{ij},\label{ruv}
\end{gather}
where $\de_{i<j}$ or $\de_{i>j}$ equals $1$ if the subscript inequality
holds, and $0$ otherwise, and
\begin{gather}\label{dij}
d_{ij}(u,v)=\begin{cases} \big(q-q^{-1}\big)\tss\xi\tss v\tss q^{\bi-\bj}\ts\ve_i\tss\ve_j& \text{for}\ i<j,\\
\big(q-q^{-1}\big)\tss u\tss q^{\bi-\bj}\ts\ve_i\tss\ve_j& \text{for}\ i>j,\\
\big(1-q^{-1}\big)(u+v\tss q\tss\xi)& \text{for}\ i=j, \ i\ne i',\\
\big(1-q^{-1}\big)(u\tss q+v\tss\xi)& \text{for}\ i=j=i'.
\end{cases}
\end{gather}
Observe that by setting $d_{ij}(u,v)\equiv 0$ we recover the trigonometric $R$-matrix for type~$A$.
Note also that if $N=2n$, then the equality $i=i'$ is impossible, so that the last case in the definition of $d_{ij}(u,v)$
does not occur.

The {\em quantum affine algebra} $\U^{R}_q(\whg)$ (with the trivial central charge)
is generated by the
elements ${l}^{\pm}_{ij}[\mp m]$ with $1\leqslant i,j\leqslant N$ and $m\in \ZZ_{+}$
subject to the following defining relations. We have\footnote{The condition $i<j$
was erroneously replaced by the opposite inequality in \cite{jlm:ib-c, jlm:ib-bd}.}
\begin{gather}\label{lpmo}
{l}_{ij}^{+}[0]={l}^{-}_{ji}[0]=0\quad\text{for}\quad i<j
\Fand {l}^{+}_{ii}[0]\ts {l}_{ii}^{-}[0]={l}_{ii}^{-}[0]\ts{l}^{+}_{ii}[0]=1,
\end{gather}
while the remaining relations will be written in terms of the formal power series
\begin{gather}
{l}^{\pm}_{ij}(u)=\sum_{m=0}^{\infty}{l}^{\pm}_{ij}[\mp m]\ts u^{\pm m}
\end{gather}
combined into the respective matrices
\begin{gather}
{L}^{\pm}(u)=\sum\limits_{i,j=1}^{N} e_{ij}\ot{l}_{ij}^{\pm}(u)\in
 \End\CC^{N}\ot\U^{R}_q(\whg)\big[\big[u,u^{-1}\big]\big].
\end{gather}
Consider the tensor product algebra $\End\CC^{N}\ot\End\CC^{N}\ot \U^{R}_q(\whg)$
and introduce the series with coefficients in this algebra by
\begin{gather}
{L}^{\pm}_1(u)=\sum\limits_{i,j=1}^{N} e_{ij}\ot 1\ot{l}_{ij}^{\pm}(u)
\Fand
{L}^{\pm}_2(u)=\sum\limits_{i,j=1}^{N} 1 \ot e_{ij}\ot{l}_{ij}^{\pm}(u).
\end{gather}
The defining relations then take the form
\begin{gather}\label{rllss}
{R}(u,v)L^{\pm}_{1}(u)L^{\pm}_2(v)=L^{\pm}_2(v)L^{\pm}_{1}(u){R}(u,v),\\
{R}(u,v)L^{+}_1(u)L^{-}_2(v)=L^{-}_2(v)L^{+}_1(u){R}(u,v),\label{rllmp}
\end{gather}
together with the relations
\begin{gather}\label{unitary}
{L}^{\pm}(u)D{L}^{\pm}(u\tss\xi)^{\tra}D^{-1}=1,
\end{gather}
where $D$ is the diagonal matrix
\begin{gather*}
D=\diag\ts\big[q^{\overline{1}},\dots,q^{\overline{N}}\tss\big].
\end{gather*}

\subsubsection{Isomorphisms}

By applying the Gauss decomposition to ${L}^{+}(u)$ and ${L}^{-}(u)$ introduce matrices
\begin{gather*}
F^{\pm}(u)=\begin{bmatrix}
1&0&\dots&0\ts\\
f^{\pm}_{21}(u)&1&\dots&0\\
\vdots&\vdots&\ddots&\vdots\\
f^{\pm}_{N\ts1}(u)&f^{\pm}_{N\ts2}(u)&\dots&1
\end{bmatrix},
\qquad
E^{\pm}(u)=\begin{bmatrix}
\ts1&e^{\pm}_{12}(u)&\dots&e^{\pm}_{1\ts N}(u)\ts\\
\ts0&1&\dots&e^{\pm}_{2\ts N}(u)\\
\vdots&\vdots&\ddots&\vdots\\
0&0&\dots&1
\end{bmatrix},
\end{gather*}
and $H^{\pm}(u)=\diag\ts\big[h^{\pm}_1(u),\dots,h^{\pm}_{N}(u)\big]$,
such that
\begin{gather}\label{gaussdec}
L^{\pm}(u)=F^{\pm}(u)H^{\pm}(u)E^{\pm}(u).
\end{gather}
Their entries are found by the quasideterminant formulas~\cite{gr:tn}:
\begin{gather}\label{hmqua}
h^{\pm}_i(u)=\begin{vmatrix} l^{\pm}_{1\tss 1}(u)&\dots&l^{\pm}_{1\ts i-1}(u)&l^{\pm}_{1\tss i}(u)\\
 \vdots&\ddots&\vdots&\vdots\\
 l^{\pm}_{i-1\ts 1}(u)&\dots&l^{\pm}_{i-1\ts i-1}(u)&l^{\pm}_{i-1\ts i}(u)\\[0.2em]
 l^{\pm}_{i\tss 1}(u)&\dots&l^{\pm}_{i\ts i-1}(u)&\boxed{l^{\pm}_{i\tss i}(u)}\\
 \end{vmatrix},\qquad i=1,\dots,N,
\end{gather}
whereas
\begin{gather}\label{eijmlqua}
e^{\pm}_{ij}(u)=h^{\pm}_i(u)^{-1}\ts\begin{vmatrix} l^{\pm}_{1\tss 1}(u)&\dots&
l^{\pm}_{1\ts i-1}(u)&l^{\pm}_{1\ts j}(u)\\
 \vdots&\ddots&\vdots&\vdots\\
 l^{\pm}_{i-1\ts 1}(u)&\dots&l^{\pm}_{i-1\ts i-1}(u)&l^{\pm}_{i-1\ts j}(u)\\[0.2em]
 l^{\pm}_{i\tss 1}(u)&\dots&l^{\pm}_{i\ts i-1}(u)&\boxed{l^{\pm}_{i\tss j}(u)}\\
 \end{vmatrix}
\end{gather}
and
\begin{gather*}
f^{\pm}_{ji}(u)=\begin{vmatrix} l^{\pm}_{1\tss 1}(u)&\dots&l^{\pm}_{1\ts i-1}(u)&l^{\pm}_{1\tss i}(u)\\
 \vdots&\ddots&\vdots&\vdots\\
 l^{\pm}_{i-1\ts 1}(u)&\dots&l^{\pm}_{i-1\ts i-1}(u)&l^{\pm}_{i-1\ts i}(u)\\[0.2em]
 l^{\pm}_{j\ts 1}(u)&\dots&l^{\pm}_{j\ts i-1}(u)&\boxed{l^{\pm}_{j\tss i}(u)}\\
 \end{vmatrix}\ts h^{\pm}_i(u)^{-1}
\end{gather*}
for $1\leqslant i<j\leqslant N$.
Set
\begin{gather*}
X^{+}_i(u)=e^{+}_{i,i+1}(u)-e_{i,i+1}^{-}(u),
\qquad X^{-}_i(u)=f^{+}_{i+1,i}(u)-f^{-}_{i+1,i}(u),
\end{gather*}
for $i=1,\dots, n-1$, and
\begin{gather*}
X^{+}_n(u)=\begin{cases}
 e^{+}_{n,n+1}(u)-e_{n,n+1}^{-}(u)&\text{for types $B_n$ and $C_n$}, \\
 e^{+}_{n-1,n+1}(u)-e_{n-1,n+1}^{-}(u)&\text{for type $D_n$},
 \end{cases}\\
X^{-}_n(u)=\begin{cases}
 f^{+}_{n+1,n}(u)-f^{-}_{n+1,n}(u)&\text{for types $B_n$ and $C_n$}, \\
 f^{+}_{n+1,n-1}(u)-f^{-}_{n+1,n-1}(u)&\text{for type $D_n$}.
 \end{cases}
\end{gather*}
Combine the generators $x^{\pm}_{i,m}$ of the algebra $\U_q(\whg)$ into the series
\begin{gather*}
x^{\pm}_{i}(u)=\sum_{m\in\ZZ}x^{\pm}_{i,m}\ts u^{-m}.
\end{gather*}
By the Main Theorems of \cite{jlm:ib-c} and \cite{jlm:ib-bd},
the maps
\begin{gather}
x^{\pm}_{i}(u) \mapsto \big(q_i-q_i^{-1}\big)^{-1}X^{\pm}_i\big(uq^i\big),\nonumber\\
\psi_{i}(u) \mapsto h^{-}_{i+1}\big(uq^i\big)\ts h^{-}_{i}\big(uq^i\big)^{-1},\nonumber\\
\varphi_{i}(u) \mapsto h^{+}_{i+1}\big(uq^i\big)\ts h^{+}_{i}\big(uq^i\big)^{-1},\label{isoma}
\end{gather}
for $i=1,\dots,n-1$, and
\begin{gather*}
x^{\pm}_{n}(u)\mapsto
\begin{cases}\big(q_n-q_n^{-1}\big)^{-1}[2]_{q_n}^{-1/2}X^{\pm}_n\big(uq^{n}\big)
 &\text{for type $B_n$}, \\
\big(q_n-q_n^{-1}\big)^{-1}X^{\pm}_n\big(uq^{n+1}\big)
 &\text{for type $C_n$}, \\
\big(q_n-q_n^{-1}\big)^{-1}X^{\pm}_n\big(uq^{n-1}\big) &\text{for type $D_n$},
\end{cases}
\\
\psi_{n}(u)\mapsto
\begin{cases}
h^{-}_{n+1}\big(uq^{n}\big)\ts h^{-}_{n}\big(uq^{n}\big)^{-1} &\text{for type $B_n$}, \\
h^{-}_{n+1}\big(uq^{n+1}\big)\ts h^{-}_{n}\big(uq^{n+1}\big)^{-1} &\text{for type $C_n$}, \\
h^{-}_{n+1}\big(uq^{n-1}\big)\ts h^{-}_{n-1}\big(uq^{n-1}\big)^{-1} &\text{for type $D_n$},
\end{cases}
\\
\varphi_{n}(u)\mapsto
\begin{cases}
h^{+}_{n+1}\big(uq^{n}\big)\ts h^{+}_{n}\big(uq^{n}\big)^{-1} &\text{for type $B_n$}, \\
h^{+}_{n+1}\big(uq^{n+1}\big)\ts h^{+}_{n}\big(uq^{n+1}\big)^{-1} &\text{for type $C_n$}, \\
h^{+}_{n+1}\big(uq^{n-1}\big)\ts h^{+}_{n-1}\big(uq^{n-1}\big)^{-1} &\text{for type $D_n$},
\end{cases}
\end{gather*}
define an isomorphism $\U_q(\whg)\to \U^{R}_q(\whg)$.

\subsubsection{Central elements of the extended quantum affine algebra}\label{subsev:zu}

The {\em extended quantum affine algebra} is defined by the same
presentation as the algebra $\U^{R}_q(\whg)$, except that the relation~\eqref{unitary} is omitted. It was shown in \cite{jlm:ib-c,jlm:ib-bd},
that this algebra can be explicitly described in terms of the Gaussian generators
by producing complete sets of relations. We will
denote the extended algebra by $\U^{\ext}_q(\whg)$ and identify
its $R$-matrix and Gaussian presentations.

The maps described above can be understood as an embedding $\iota\colon \U_q(\whg)\hra \U^{\ext}_q(\whg)$
so that we can regard $\U_q(\whg)$ as a subalgebra of $\U^{\ext}_q(\whg)$.
This subalgebra can also be described with the use of the multiplication
automorphisms
\begin{gather}\label{muf}
\mu_f\colon \ \U^{\ext}_q(\whg)\to \U^{\ext}_q(\whg),\qquad L^{\pm}(u)\mapsto f^{\pm}(u)\tss L^{\pm}(u),
\end{gather}
where
\begin{gather}\label{furu}
f^{\pm}(u)=\sum_{m=0}^{\infty}f^{\pm}[\mp m]\tss u^{\pm m},\qquad f^{\pm}[\mp m]\in\CC,
\qquad f^{+}[0]\tss f^{-}[0]=1.
\end{gather}
Namely, the image of $\U_q(\whg)$ under the embedding $\iota$ consists of the elements in
$\U^{\ext}_q(\whg)$ which are fixed by all automorphisms of the form~\eqref{muf}.

All coefficients of the series $z^{\pm}(u)$ given by
\begin{gather*}
z^{\pm}(u)=\begin{cases}\displaystyle\prod_{i=1}^{n}{h}^{\pm}_{i}\big(u\xi q^{2i}\big)^{-1}
\prod_{i=1}^{n}{h}^{\pm}_{i}\big(u\xi q^{2i-2}\big)\ts\cdot {h}^{\pm}_{n+1}(u)\ts{h}^{\pm}_{n+1}(uq)
&\text{for}\ N=2n+1,\\
\displaystyle\prod_{i=1}^{n-1}{h}^{\pm}_{i}\big(u\xi q^{2i}\big)^{-1}
\prod_{i=1}^{n}{h}^{\pm}_{i}\big(u\xi q^{2i-2}\big)\cdot{h}^{\pm}_{n+1}(u)
& \text{for} \ N=2n,
\end{cases}
\end{gather*}
belong to the center of the algebra $\U^{\ext}_q(\whg)$.
If $N=2n$ then the constant terms of the series~$z^{\pm}(u)$ are the
central elements $z^{\pm}[0]=l^{\pm}_{nn}[0]\ts l_{n'n'}^{\pm}[0]$. In this case
we will extend this algebra
by adjoining the square roots $z^{\pm}[0]^{1/2}$.
Then in all three cases there exist power series
$\ze^{\pm}(u)$ with coefficients in the center of $\U^{\ext}_q(\whg)$
such that $\ze^{\pm}(u)\ts\ze^{\pm}(u\tss\xi)=z^{\pm}(u)$.
Under the automorphism~\eqref{muf} we have
\begin{gather*}
\mu_f\colon \ \ze^{\pm}(u)\mapsto f^{\pm}(u)\tss \ze^{\pm}(u).
\end{gather*}
This implies that the coefficients of the entries of the matrices
$\ze^{\pm}(u)^{-1}\tss L^{\pm}(u)$ belong to the subalgebra
$\U_q(\whg)\subset \U^{\ext}_q(\whg)$. Therefore, if
$\Cc$ denotes the
subalgebra of $\U^{\ext}_q(\whg)$ generated by the coefficients
of the series $\ze^{\pm}(u)$, then we
have the tensor product decomposition
\begin{gather*}
\U^{\ext}_q(\whg) = \U_q(\whg)\otimes \Cc,
\end{gather*}
assuming that the algebra $\U_q(\whg)$ is extended by adjoining
square roots in types $C_n$ and $D_n$ as in Remark~\ref{rem:ext}.
In the algebra $\U^{\ext}_q(\whg)$ we have
\begin{gather}\label{unize}
{L}^{\pm}(u)D{L}^{\pm}(u\ts\xi)^{\tra}D^{-1}=z^{\pm}(u),
\end{gather}
so that the subalgebra $\U_q(\whg)$ can also be regarded as the quotient of $\U^{\ext}_q(\whg)$
by the relations $z^{\pm}(u)=1$. The following relations are implied by \eqref{unize},
and they were essentially derived in Section~4.5 of \cite{jlm:ib-c} and \cite{jlm:ib-bd}.

\begin{Lemma}\label{lem:gaurell}
In the algebra $\U^{\ext}_q(\whg)$ we have
\begin{gather*}
h^{\pm}_1(u\xi)\tss h^{\pm}_{1'}(u)=z^{\pm}(u),\label{hone}\\
h^{\pm}_i(u\xi q^{2i})\tss h^{\pm}_{i'}(u)=h^{\pm}_{i+1}(u\xi q^{2i})\tss h^{\pm}_{(i+1)'}(u),
\end{gather*}
where $i=1,\dots,n$ for $N=2n+1$, and $i=1,\dots,n-1$ for $N=2n$.
\end{Lemma}

\begin{Remark}\label{rem:trans}
Note that for the parameters $d_{ij}(u,v)$ defined in
\eqref{dij} we have $d_{j'i'}(u,v)=d_{ij}(u,v)$. Therefore
the $R$-matrix \eqref{ruv} possesses the symmetry property
\begin{gather}\label{rtra}
R^{T_1T_2}(u,v)=R_{21}(u,v),
\end{gather}
where $R_{21}(u,v)=PR(u,v)P$, while
$T$ denotes the standard matrix
transposition with $e_{ij}^{T}=e_{ji}$ and $T_a$ is the partial transposition
applied to the $a$-th copy of the endomorphism algebra $\End\CC^N$.
We can use the $R$-matrix $R_{21}(u,v)$ instead of $R(u,v)$
to define the {\em extended quantum affine algebra} $\wt \U^{\ext}_q(\whg)$
in a way similar to $\U^{\ext}_q(\whg)$, by using the relations
\begin{gather*}
{R}_{21}(u,v)\wt L^{\pm}_{1}(u)\wt L^{\pm}_2(v)=\wt L^{\pm}_2(v)\wt L^{\pm}_{1}(u){R}_{21}(u,v),\\
{R}_{21}(u,v)\wt L^{+}_1(u)\wt L^{-}_2(v)=\wt L^{-}_2(v)\wt L^{+}_1(u){R}_{21}(u,v),
\end{gather*}
where we impose the opposite zero mode conditions
\begin{gather*}
{\wt l}_{ij}^{+}[0]={\wt l}^{-}_{ji}[0]=0\quad\text{for}\quad i>j
\Fand {\wt l}^{+}_{ii}[0]\ts {\wt l}_{ii}^{-}[0]={\wt l}_{ii}^{-}[0]\ts{\wt l}^{+}_{ii}[0]=1.
\end{gather*}
The symmetry property \eqref{rtra} implies that the mapping
\begin{gather*}
\wt L^{\pm}(u)\mapsto L^{\pm}(u)^T
\end{gather*}
defines an anti-isomorphism $\wt \U^{\ext}_q(\whg)\to \U^{\ext}_q(\whg)$.
Using the definition of quasideterminants, we obtain the following formulas
for the images of the respective Gaussian generators
\begin{gather*}
\wt h^{\pm}_{i}(u)\mapsto h^{\pm}_{i}(u),\qquad
\wt e^{\ts\pm}_{ij}(u)\mapsto f^{\pm}_{ji}(u)\Fand
\wt f^{\ts\pm}_{ij}(u)\mapsto e^{\pm}_{ji}(u).
\end{gather*}
\end{Remark}

\subsubsection{Hopf algebra structure}\label{subsubsec:ha}

The quantum affine algebra $\U^{\ext}_q(\whg)$ possesses a Hopf algebra structure defined by the coproduct
\begin{gather}\label{Delta}
\Delta\colon \ l^{\pm}_{ij}(u)\mapsto \sum_{k=1}^N l^{\pm}_{ik}(u)\ot l^{\pm}_{kj}(u),
\end{gather}
the antipode
\begin{gather}\label{S}
S\colon \ L^{\pm}(u)\mapsto L^{\pm}(u)^{-1}
\end{gather}
and the counit
\begin{gather}\label{cou}
\ve\colon \ L^{\pm}(u)\mapsto 1.
\end{gather}

\begin{Proposition}\label{prop:zuho}
In the Hopf algebra $\U^{\ext}_q(\whg)$ we have
\begin{gather*}
\Delta\colon \ z^{\pm}(u)\mapsto z^{\pm}(u)\ot z^{\pm}(u)
\end{gather*}
and
\begin{gather}\label{sant}
S\colon \ z^{\pm}(u)\mapsto z^{\pm}(u)^{-1}.
\end{gather}
In particular, $\U_q(\whg)$ is a Hopf subalgebra of $\U^{\ext}_q(\whg)$.
Moreover,
\begin{gather}\label{sqs}
S^2\colon \ L^{\pm}(u)\mapsto \frac{z^{\pm}(u)}{z^{\pm}(u\tss\xi)}\tss L^{\pm}\big(u\tss \xi^2\big).
\end{gather}
\end{Proposition}

\begin{proof}
The formulas for the images of the series $z^{\pm}(u)$ under the maps $\Delta$ and $S$ follow
easily from the definition of $z^{\pm}(u)$ and the Hopf algebra axioms.
For the proof of \eqref{sqs} use the relation
\begin{gather*}
L^{\pm}(u)^{-1}=z^{\pm}(u)^{-1}\tss D{L}^{\pm}(u\tss\xi)^{\tra}D^{-1}
\end{gather*}
and apply \eqref{sant}. Furthermore,
for the power series
$\ze^{\pm}(u)$ we have
\begin{gather*}
\Delta\colon \ \ze^{\pm}(u)\mapsto \ze^{\pm}(u)\ot \ze^{\pm}(u),
\end{gather*}
and so
\begin{gather*}
\Delta\big(\U_q(\whg)\big)\subset \U_q(\whg)\ot \U_q(\whg),\qquad S\big(\U_q(\whg)\big)\subset \U_q(\whg),
\end{gather*}
thus proving that $\U_q(\whg)$ is a Hopf subalgebra of $\U^{\ext}_q(\whg)$.
\end{proof}

\subsubsection{Consistency with the triangular decomposition}
\label{subsubsec:td}

Denote by $\U_q(\whg)^+$ (respectively, $\U_q(\whg)^-$) the subalgebra of $\U_q(\whg)$
generated by the elements $x_{i,m}^+$ (respectively, $x_{i,m}^-$), and denote by $\U_q(\whg)^0$
the subalgebra generated by $k_i^{\pm 1}$ and $a_{i,l}$ together with
the additional elements in types $C_n$ and $D_n$, introduced in Remark~\ref{rem:ext}.
The multiplication map provides the triangular decomposition isomorphism
\begin{gather*}
\U_q(\whg)^-\ot \U_q(\whg)^0\ot \U_q(\whg)^+\cong \U_q(\whg),
\end{gather*}
as proved in \cite{b:bg}; see also \cite{h:rq} for a generalization
to quantum affinizations of symmetrizable
Kac--Moody algebras.

Here we aim to establish a key property of the Gauss decomposition by showing that
it is consistent with the triangular decomposition (see Proposition~\ref{prop:consi} below).
We will rely on a few relations for the Gaussian generators
described in the following lemmas.

We will use a standard notation $[x,y]_q=x\tss y-q\tss y\tss x$
for $q$-commutators and begin
by proving some $q$-commutator formulas; cf.\ \cite[Lemma~5.6]{kpt:ob}.

\begin{Lemma}\label{lem:kpt}For any $k<i<j<k'$ such that $i\ne j'$ we have
\begin{gather*}
\big[e^{\pm}_{k\tss i}(u),e^{-}_{i\tss j}[0]\big]_{q} =\big(1 - q^2\big)\ts e^{\pm}_{k\tss j}(u).
\end{gather*}
\end{Lemma}

\begin{proof}Due to the consistency property of Gauss decompositions for subalgebras as stated
in Proposition~4.2 in \cite{jlm:ib-c} and \cite{jlm:ib-bd}, we may assume without loss
of generality, that $k=1$.
By taking matrix entries in \eqref{rllss} and
\eqref{rllmp}, write the defining relations in terms of the series $l^{\pm}_{ij}(u)$
to get
\begin{gather*}
\big(u\tss q^{\delta_{ij}} - v\tss q^{-\delta_{ij}}\big)\ts
l^{\pm}_{ia}(u)\ts l^{\pm}_{jb}(v)
+ \big(q - q^{-1}\big) \ts (u\ts\delta_{i<j} +v\ts \delta_{i>j})\ts
l^{\pm}_{ja}(u)\ts l^{\pm}_{ib}(v)\\
\qquad\quad{} -\de_{i\tss j'}\ts\frac{u-v}{u-v\tss\xi}\ts \sum_{k=1}^N \ts
d_{jk}(u,v)\tss l^{\pm}_{k'a}(u)\ts l^{\pm}_{kb}(v)\\
\qquad{}=\big(u\tss q^{\delta_{ab}} - v\tss q^{-\delta_{ab}}\big)
\ts l^{\pm}_{jb}(v)\ts l^{\pm}_{ia}(u)
{}+ \big(q - q^{-1}\big) \ts (u\ts \delta_{a>b}
+ v\ts \delta_{a<b})\ts l^{\pm}_{ja}(v)\ts l^{\pm}_{ib}(u)\\
\qquad\quad{} -\de_{a\tss b'}\ts\frac{u-v}{u-v\tss\xi}\ts \sum_{c=1}^N \ts
d_{cb}(u,v)\tss l^{\pm}_{jc}(v)\ts l^{\pm}_{ic'}(u),
\end{gather*}
and
\begin{gather*}
\big(u\tss q^{\delta_{ij}} - v\tss q^{-\delta_{ij}}\big)\ts
l^{+}_{ia}(u)\ts l^{-}_{jb}(v)
 + \big(q - q^{-1}\big) \ts (u\ts\delta_{i<j} +v\ts \delta_{i>j})\ts
l^{+}_{ja}(u)\ts l^{-}_{ib}(v)\\
\qquad\quad{} -\de_{i\tss j'}\ts\frac{u-v}{u-v\tss\xi}\ts \sum_{k=1}^N \ts
d_{jk}(u,v)\tss l^{+}_{k'a}(u)\ts l^{-}_{kb}(v)\\
\qquad{}=\big(u\tss q^{\delta_{ab}} - v\tss q^{-\delta_{ab}}\big)
\ts l^{-}_{jb}(v)\ts l^{+}_{ia}(u)
 + \big(q - q^{-1}\big) \ts (u\ts \delta_{a>b}
+ v\ts \delta_{a<b})\ts l^{-}_{ja}(v)\ts l^{+}_{ib}(u)\\
\qquad\quad{} -\de_{a\tss b'}\ts\frac{u-v}{u-v\tss\xi}\ts \sum_{c=1}^N \ts
d_{cb}(u,v)\tss l^{-}_{jc}(v)\ts l^{+}_{ic'}(u),
\end{gather*}
where $d_{ij}(u,v)$ are defined in~\eqref{dij}. If $1<i<j$ and $j\ne i'$, then
they give
\begin{gather*}
(u-v)\tss l^{\pm}_{1\tss i}(u)\tss l^{-}_{i\tss j}(v)
+\big(q - q^{-1}\big)\ts u\tss
l^{\pm}_{i\tss i}(u)\ts l^{-}_{1\tss j}(v)\\
\qquad{}=(u-v)\tss l^{-}_{i\tss j}(v)\ts l^{\pm}_{1\tss i}(u)
+\big(q - q^{-1}\big)\ts v\tss
l^{-}_{i\tss i}(v)\ts l^{\pm}_{1\tss j}(u).
\end{gather*}
By comparing the coefficients of $v$ on both sides, we come to the relations
\begin{gather}\label{loni}
l^{\pm}_{1\tss i}(u)\tss l^{-}_{i\tss j}[0]=l^{-}_{i\tss j}[0]\ts l^{\pm}_{1\tss i}(u)
-\big(q - q^{-1}\big)\tss l^{-}_{i\tss i}[0]\ts l^{\pm}_{1\tss j}(u).
\end{gather}
Similarly, assuming that $1<i\leqslant m$ and $m\ne 1'$, we get
\begin{gather*}
(u-v)\tss l^{\pm}_{1\tss 1}(u)\tss l^{-}_{i\tss m}(v)
+\big(q - q^{-1}\big)\ts u\tss
l^{\pm}_{i\tss 1}(u)\ts l^{-}_{1\tss m}(v)\\
\qquad{}=(u-v)\tss l^{-}_{i\tss m}(v)\ts l^{\pm}_{1\tss 1}(u)
+\big(q - q^{-1}\big)\ts v\tss
l^{-}_{i\tss 1}(v)\ts l^{\pm}_{1\tss m}(u),
\end{gather*}
which yields
\begin{gather}\label{lonene}
l^{\pm}_{1\tss 1}(u)\tss l^{-}_{i\tss m}[0]=l^{-}_{i\tss m}[0]\ts l^{\pm}_{1\tss 1}(u).
\end{gather}
Since $l^{\pm}_{1\tss i}(u)=l^{\pm}_{1\tss 1}(u)\tss e^{\pm}_{1\tss i}(u)$,
together with relation \eqref{loni} this implies
\begin{gather}\label{eul}
e^{\pm}_{1\tss i}(u)\tss l^{-}_{i\tss j}[0]=l^{-}_{i\tss j}[0]\ts e^{\pm}_{1\tss i}(u)
-\big(q - q^{-1}\big)\tss l^{-}_{i\tss i}[0]\ts e^{\pm}_{1\tss j}(u).
\end{gather}
Applying the defining relations again, for any $1<i<1'$ we get
\begin{gather*}
(u-v)\tss l^{\pm}_{1\tss i}(u)\tss l^{-}_{i\tss i}(v)
+\big(q - q^{-1}\big)\ts u\tss
l^{\pm}_{i\tss i}(u)\ts l^{-}_{1\tss i}(v)
=\big(u\tss q-v\tss q^{-1}\big)\tss l^{-}_{i\tss i}(v)\ts l^{\pm}_{1\tss i}(u).
\end{gather*}
Comparing the coefficients of $v$ on both sides, we come to
\begin{gather*}
l^{\pm}_{1\tss i}(u)\tss l^{-}_{i\tss i}[0]=q^{-1}\tss l^{-}_{i\tss i}[0]\ts l^{\pm}_{1\tss i}(u)
\end{gather*}
which implies
\begin{gather}\label{epmi}
e^{\pm}_{1\tss i}(u)\tss l^{-}_{i\tss i}[0]=q^{-1}\tss l^{-}_{i\tss i}[0]\ts e^{\pm}_{1\tss i}(u).
\end{gather}
Hence, using this together with the decomposition $l^{-}_{i\tss j}[0]=l^{-}_{i\tss i}[0]e^-_{ij}[0]$,
we derive from \eqref{eul} that
\begin{gather*}
q^{-1}\tss e^{\pm}_{1\tss i}(u)\tss e^{-}_{i\tss j}[0]=e^{-}_{i\tss j}[0]\ts e^{\pm}_{1\tss i}(u)
-\big(q - q^{-1}\big)\ts e^{\pm}_{1\tss j}(u).
\end{gather*}
We can write this as the $q$-commutator relation
\begin{gather*}
\big[e^{\pm}_{1\tss i}(u),e^{-}_{i\tss j}[0]\big]_q =\big(1 - q^2\big)\ts e^{\pm}_{1\tss j}(u),
\end{gather*}
as required.
\end{proof}

\begin{Lemma}\label{lem:nthro} For $n\geqslant 2$ we have
\begin{gather*}
\big[e^{\pm}_{1\tss 2'}(u),e^{-}_{2' 1'}[0]\big]
=\big(1 - q^2\big)\ts \big(e^{\pm}_{1\tss 1'}(u)+e^{\pm}_{1\tss 2}(u)e^{\pm}_{1\tss 2'}(u)\big).
\end{gather*}
\end{Lemma}

\begin{proof}The defining relations give
\begin{gather*}
(u- v)\ts
l^{\pm}_{11}(u)\ts l^{-}_{2' 1'}(v)
 + \big(q - q^{-1}\big)\ts u\ts
l^{\pm}_{2'1}(u)\ts l^{-}_{11'}(v)\\
{}=(u- v)
\ts l^{-}_{2' 1'}(v)\ts l^{\pm}_{11}(u)
+ \big(q - q^{-1}\big)\ts v\ts l^{-}_{2'1}(v)\ts l^{\pm}_{11'}(u)
-\frac{u-v}{u-v\tss\xi}\ts \sum_{c=1}^N \ts
d_{c\tss 1'}(u,v)\tss l^{-}_{2'c}(v)\ts l^{\pm}_{1\tss c'}(u).
\end{gather*}
By comparing the coefficients of $v$ on both sides we come to
the relation
\begin{gather}\label{ploni}
l^{-}_{2' 1'}[0]\tss l^{\pm}_{1\tss 1}(u)=q^{-1}\tss l^{\pm}_{1\tss 1}(u)\tss l^{-}_{2' 1'}[0]
-\big(q - q^{-1}\big)\tss
l^{-}_{2' 2'}[0]\ts l^{\pm}_{1\tss 2}(u).
\end{gather}
On the other hand, taking $i=2'$ and $j=1'$ in \eqref{loni} we get
\begin{gather*}
l^{\pm}_{1\tss 2'}(u)\tss l^{-}_{2' 1'}[0]=l^{-}_{2' 1'}[0]\ts l^{\pm}_{1\tss 2'}(u)
-\big(q - q^{-1}\big)\tss l^{-}_{2' 2'}[0]\ts l^{\pm}_{1\tss 1'}(u).
\end{gather*}
Write $l^{\pm}_{1\tss i}(u)=l^{\pm}_{1\tss 1}(u)\tss e^{\pm}_{1\tss i}(u)$ for $i=1',2'$
and use~\eqref{ploni} together with~\eqref{lonene} to bring this to the form
\begin{gather*}
e^{\pm}_{1\tss 2'}(u)\tss l^{-}_{2' 1'}[0]=q^{-1}\tss l^{-}_{2' 1'}[0]\ts e^{\pm}_{1\tss 2'}(u)
-\big(q - q^{-1}\big)\tss
l^{-}_{2' 2'}[0]\ts \big(e^{\pm}_{1\tss 1'}(u)+e^{\pm}_{1\tss 2}(u)e^{\pm}_{1\tss 2'}(u)\big).
\end{gather*}
Finally, write $l^{-}_{2' 1'}[0]=l^{-}_{2' 2'}[0]\tss e^{-}_{2' 1'}[0]$
and apply \eqref{epmi} with $i=2'$.
\end{proof}

\begin{Lemma}\label{lem:gainv}
For any $i<j<i'$
we have the relations
\begin{gather*}
e^{\pm}_{ij}\big(u\xi q^{2i}\big)=q^{\bj-\bi+1}\tss\sum_{s=0}^{j-i-1}\tss(-1)^{s+1}
\sum_{j'=a_0<a_1<\dots<a_{s+1}=i'}e^{\pm}_{a_0a_1}(u)\tss e^{\pm}_{a_1a_2}(u)\cdots
e^{\pm}_{a_sa_{s+1}}(u).
\end{gather*}
\end{Lemma}

\begin{proof}We will rely on Proposition~4.2 in \cite{jlm:ib-c} and \cite{jlm:ib-bd},
which allows us to reduce the proof to the case $i=1$.
Working in the algebra $\U^{\ext}_q(\whg)$,
write~\eqref{unize} in the form
\begin{gather*}
D{L}^{\pm}(u\ts\xi)^{\tra}D^{-1}={L}^{\pm}(u)^{-1}z^{\pm}(u)
\end{gather*}
and take the $(j',1')$ entries on both sides. Using the Gauss decompositions \eqref{gaussdec}
and
\begin{gather*}
L^{\pm}(u)^{-1}=E^{\pm}(u)^{-1}H^{\pm}(u)^{-1}F^{\pm}(u)^{-1},
\end{gather*}
we obtain
\begin{gather*}
q^{\bar 1-\bj}\ts h^{\pm}_1(u\tss\xi)\tss e^{\pm}_{1j}(u\tss\xi)
=\wh e^{\ts\pm}_{j'1'}(u)\tss h^{\pm}_{1'}(u)^{-1}\tss z^{\pm}(u),
\end{gather*}
where we used the notation $\wh e^{\ts\pm}_{ij}(u)$ for the $(i,j)$ entry of the matrix
$E^{\pm}(u)^{-1}$. Now apply \eqref{hone} to replace $h^{\pm}_{1'}(u)^{-1}\tss z^{\pm}(u)$
with $h^{\pm}_1(u\tss\xi)$ and note that
\begin{gather}\label{hexi}
h^{\pm}_1(u\tss\xi)\tss e^{\pm}_{1j}(u\tss\xi)=q^{-1}\tss e^{\pm}_{1j}\big(u\tss\xi\tss q^2\big)\tss
h^{\pm}_1(u\tss\xi).
\end{gather}
Indeed, by the defining relations,
\begin{gather*}
\big(u\tss q-v\tss q^{-1}\big)\tss l^{\pm}_{1\tss 1}(u)\tss l^{\pm}_{1\tss j}(v)
=(u-v)\tss l^{\pm}_{1\tss j}(v)\ts l^{\pm}_{1\tss 1}(u)
+\big(q - q^{-1}\big)\ts v\tss l^{\pm}_{1\tss 1}(v)\ts l^{\pm}_{1\tss j}(u).
\end{gather*}
By setting $v=u\tss q^2$ we get
\begin{gather*}
l^{\pm}_{1\tss j}(u\tss q^2)\ts l^{\pm}_{1\tss 1}(u)= q\ts l^{\pm}_{1\tss 1}\big(u\tss q^2\big)\ts l^{\pm}_{1\tss j}(u).
\end{gather*}
Since $l^{\pm}_{1\tss j}(u)=h^{\pm}_1(u)\tss e^{\pm}_{1\tss j}(u)$,
\eqref{hexi} follows.
It remains to apply the formula
\begin{gather*}
\wh e^{\ts\pm}_{j'1'}(u)=
\sum_{s=0}^{j-2}\tss(-1)^{s+1}
\sum_{j'=a_0<a_1<\dots<a_{s+1}=1'}e^{\pm}_{a_0a_1}(u)\tss e^{\pm}_{a_1a_2}(u)\cdots
e^{\pm}_{a_sa_{s+1}}(u)
\end{gather*}
for the $(j',1')$ entries
of the inverse matrix $E^{\pm}(u)^{-1}$.
\end{proof}

\begin{Proposition}\label{prop:consi}
The images of the coefficients of all series $e^{\pm}_{ij}(u)$
{\rm(}respectively, $f^{\pm}_{ji}(u)${\rm)} with $1\leqslant i<j\leqslant N$
under the isomorphism $\U^R_q(\whg)\to \U_q(\whg)$ belong
to the subalgebra $\U_q(\whg)^+$ {\rm(}respectively, $\U_q(\whg)^-${\rm)}.
The images of the coefficients of all series $h^{\pm}_i(u)$
with $i=1,\dots,N$ belong
to the subalgebra $\U_q(\whg)^0$.
\end{Proposition}

\begin{proof}
We will identify the coefficients of all the series with their images under the
isomorphism $\U^R_q(\whg)\to \U_q(\whg)$.
We start by verifying the claim for the series $h^{+}_i(u)$; the argument
for $h^{-}_i(u)$ is exactly the same.
We have
\begin{gather*}
h^{+}_i(u)=\vp_1\big(u\tss q^{-1}\big)\tss \vp_2\big(u\tss q^{-2}\big)\cdots \vp_{i-1}\big(u\tss q^{-i+1}\big)\tss h^{+}_1(u),
\qquad i=2,\dots,n.
\end{gather*}
This relation is also valid for $i=n+1$ in type $B_n$, whereas
\begin{gather*}
h^{+}_{n+1}(u)=\vp_{n}\big(u\tss q^{-n-1}\big)\tss h^{+}_{n}(u)\Fand
h^{+}_{n+1}(u)=\vp_{n}\big(u\tss q^{-n+1}\big)\tss h^{+}_{n-1}(u)
\end{gather*}
for types $C_n$ and $D_n$, respectively.
Now substitute these expressions
into the formula for the series $z^+(u)$ given in Section~\ref{subsev:zu},
so that the relation $z^+(u)=1$ would give an equation for the coefficients of the series
$h^{+}_1(u)$. For type $B_n$ it reads
\begin{gather*}
h^{+}_1(u)\tss h^{+}_1(u\xi)\prod_{i=1}^n\vp_i\big(uq^{-i}\big)\tss\vp_i\big(u\xi q^{i}\big)=1.
\end{gather*}
Hence, the constant term $h^{+}_{1,0}$ of $h^{+}_1(u)$ is found from
$h^{+\ts 2}_{1,0}=k_1^2\cdots k_n^2$
which together with Lemma~\ref{lem:gaurell}
implies that all coefficients of the series~$h^{+}_i(u)$
with $i=1,\dots,N$ belong to the subalgebra $\U_q(\whg)^0$. The same conclusion is reached for types~$C_n$ and~$D_n$
from the respective formulas
\begin{gather*}
h^{+}_1(u)\tss h^{+}_1(u\xi)\prod_{i=1}^{n-1}\vp_i\big(uq^{-i}\big)\tss\vp_i\big(u\xi q^{i}\big)
\cdot\vp_n\big(u\tss q^{-n-1}\big) =1
\end{gather*}
and
\begin{gather*}
h^{+}_1(u)\tss h^{+}_1(u\xi)\prod_{i=1}^{n-2}\vp_i\big(uq^{-i}\big)\tss\vp_i\big(u\xi q^{i}\big)
\cdot\vp_{n-1}\big(u\tss q^{-n+1}\big)\tss\vp_n\big(u\tss \xi\tss q^{n-1}\big) =1,
\end{gather*}
which imply the relations
$h^{+\ts 2}_{1,0}=k_1^2\cdots k_{n-1}^2k_n$
and
$h^{+\ts 2}_{1,0}=k_1^2\cdots k_{n-2}^2\tss k_{n-1}k_n$,
respectively, for the constant terms.

Now turn to the series $e^{\pm}_{ij}(u)$ and use induction on $n$.
By Proposition~4.2 in~\cite{jlm:ib-c} and~\cite{jlm:ib-bd}, the coefficients
of these series with $2\leqslant i<j\leqslant 2'$ can be regarded as
elements of the algebra~$\U_q(\wh{\g'})$ associated with the Lie algebra~$\g'$
of rank $n-1$, which belongs to the same type as~$\g$. These coefficients
coincide with the generators obtained from the Gauss decompositions
of the associated $L$-operators, so that we can apply the induction hypothesis
to suppose that the coefficients of all series $e^{\pm}_{ij}(u)$
with $2\leqslant i<j\leqslant 2'$ belong to the subalgebra~$\U_q(\whg)^+$.

On the other hand, the coefficients of the series $e^{\pm}_{12}(u)$
also belong to the subalgebra $\U_q(\whg)^+$, and so by applying Lemma~\ref{lem:kpt}
we derive that
the coefficients of the series $e^{\pm}_{1j}(u)$ with $j=2,\dots,2'$
belong to $\U_q(\whg)^+$, assuming that $n\geqslant 2$ for type $B_n$,
and $n\geqslant 3$ for types~$C_n$ and~$D_n$. Furthermore, by Lemma~\ref{lem:gainv}
we have $e^{-}_{2' 1'}[0]=-e^{-}_{12}[0]$ so that applying Lemma~\ref{lem:nthro}
we may conclude that the required property is shared by the series $e^{\pm}_{11'}(u)$.
The induction step is completed
by another application of Lemma~\ref{lem:gainv}, which implies that the coefficients
of the series $e^{\pm}_{j'1'}(u)$ with $j=2,\dots,2'$ also belong to $\U_q(\whg)^+$.

It remains to verify the induction base for the Lie algebras $\g=\oa_3,\spa_4$ and $\oa_4$. In the case $\g=\oa_3$ it follows from the identity
\begin{gather*}
e^{\pm}_{12}(u)^2=-\big(q^{1/2}+q^{-1/2}\big)\tss e^{\pm}_{11'}(u),
\end{gather*}
which is a particular case of \cite[Lemma~4.9]{jlm:ib-bd} obtained by taking
the residue at $u=q^{-2}v$ in the second formula.
In the case $\g=\spa_4$ it is sufficient to verify the identity
\begin{gather}\label{spezqij}
\big[e^{\pm}_{1\tss 2}(u),e^{-}_{2\tss 2'}[0]\big]_{q^2} =\big(1 - q^4\big)\ts e^{\pm}_{1\tss 2'}(u).
\end{gather}
By the defining relations,
\begin{gather*}
(u- v)\ts
l^{\pm}_{12}(u)\ts l^{-}_{2 2'}(v)
 + \big(q - q^{-1}\big)\ts u\ts
l^{\pm}_{22}(u)\ts l^{-}_{12'}(v)\\
{}=(u- v)
\ts l^{-}_{2 2'}(v)\ts l^{\pm}_{12}(u)
+ \big(q - q^{-1}\big)\ts v\ts l^{-}_{22}(v)\ts l^{\pm}_{12'}(u)
-\frac{u-v}{u-v\tss\xi}\ts \sum_{c=1}^4 \ts
d_{c\tss 2'}(u,v)\tss l^{-}_{2c}(v)\ts l^{\pm}_{1\tss c'}(u).
\end{gather*}
Comparing the coefficients of $v$ on both sides we derive
\begin{gather*}
l^{\pm}_{12}(u)\ts l^{-}_{2 2'}[0]=q\ts l^{-}_{2 2'}[0]\ts l^{\pm}_{12}(u)
+\big(q^{-1}-q^3\big)\ts l^{-}_{22}[0]\ts l^{\pm}_{12'}(u).
\end{gather*}
Relation~\eqref{spezqij} now follows from~\eqref{lonene} and~\eqref{epmi}.

The particular case $\g=\oa_4$ was considered in \cite[Section~4.4]{jlm:ib-bd}; it was shown
therein that $e^{\pm}_{22'}(u)=0$, which is sufficient
to complete the argument in this case. These relations essentially refer to
the algebra $\U^R_q(\whg)$ with $\g=\oa_2$, whose definition extends
to the abelian Lie algebra $\oa_2$. In this case we have $\xi=1$ and the defining
relations give\footnote{This corrects the formula used in the proof of \cite[Lemma~4.13]{jlm:ib-bd}.}
\begin{gather*}
\big(u\tss q-v\tss q^{-1}\big)\tss l^{\pm}_{11'}(u)\tss l^{\pm}_{11}(v)=\big(u\tss q^{-1}-v\tss q\big)\tss
l^{\pm}_{11}(v)\tss l^{\pm}_{11'}(u).
\end{gather*}
By setting $v=u\tss q^2$ we derive that $l^{\pm}_{11'}(u)=0$ and hence $e^{\pm}_{11'}(u)=0$.

The argument for the series $f^{\pm}_{ji}(u)$ is quite similar to that for $e^{\pm}_{ij}(u)$.
As an alternative approach, one can work with the quantum affine algebra defined
with the $R$-matrix $R_{21}(u,v)$ used instead of $R(u,v)$,
and modify the above calculations for the series $e^{\pm}_{ij}(u)$ accordingly;
then apply the anti-isomorphism introduced in Remark~\ref{rem:trans}.
\end{proof}

\subsection{Highest weight representations}\label{subsec:hw}

\begin{Definition}\label{def:hwr}
A representation $V$ of the algebra $\U^{R}_q(\whg)$ (or the extended
quantum affine algebra $\U^{\ext}_q(\whg)$)
is called a {\it highest weight representation}
if there exists a nonzero vector
$\ze\in V$ such that $V$ is generated by $\ze$
and the following relations hold:
\begin{alignat*}{2}
&l^{\pm}_{ij}(u)\ts\ze=0 \qquad &&\text{for} \quad
1\leqslant i<j\leqslant N, \qquad \text{and}\\
& l^{\pm}_{ii}(u)\ts\ze=\la^{\pm}_i(u)\ts\ze \qquad &&\text{for} \quad
i=1,\dots,N,
\end{alignat*}
for some formal series $\la^{+}_i(u)\in\CC[[u]]$ and $\la^{-}_i(u)\in\CC\big[\big[u^{-1}\big]\big]$.
The vector $\ze$ is called the {\it highest vector} of~$V$.
\end{Definition}

Note that by \eqref{lpmo}, the product of the constant terms of
$\la^{+}_i(u)$ and $\la^{-}_i(u)$ must equal~$1$.

\begin{Proposition}\label{prop:equiv}
In terms of the Gaussian generators,
the conditions in Definition~{\rm \ref{def:hwr}} are equivalent to the following:
\begin{alignat*}{2}
& e^{\pm}_{ij}(u)\ts\ze=0 \qquad &&\text{for} \quad
1\leqslant i<j\leqslant N, \qquad \text{and}\\
& h^{\pm}_{i}(u)\ts\ze=\la^{\pm}_i(u)\ts\ze \qquad &&\text{for} \quad
i=1,\dots,N.
\end{alignat*}
\end{Proposition}

\begin{proof}This is immediate from the Gauss decomposition formulas \eqref{gaussdec}--\eqref{eijmlqua}.
\end{proof}

The algebra $\U^{R}_q(\whg)$ (as well as $\U^{\ext}_q(\whg)$) admits a family of automorphisms
\begin{gather}\label{autom}
L^{+}(u)\mapsto \Sc L^{+}(u)\Fand L^{-}(u)\mapsto \Sc^{-1} L^{-}(u),
\end{gather}
parameterized by invertible diagonal matrices
$\Sc=\diag\ts [\si_1,\dots,\si_{N}]$
satisfying the conditions
\begin{gather*}
\Sc^{\ts\tra}=\Sc^{-1}=\Sc\For \Sc^{\ts\tra}=\Sc^{-1}=-\Sc.
\end{gather*}
The second condition can occur only for $N=2n$.
When applied to the isomorphic algebra $\U_q(\whg)\cong \U^{R}_q(\whg)$,
the automorphisms \eqref{autom} correspond to the sign automorphisms
considered
in \cite[Section~12.2.B]{cp:gq}.

\begin{Theorem}\label{thm:classi}\quad
\begin{enumerate}\itemsep=0pt
\item[$1.$] Any finite-dimensional irreducible representation of the algebra $\U^{R}_q(\whg)$
is a highest weight representation. Up to twisting this representation
with a suitable automorphism \eqref{autom}, its parameters satisfy the relations
\begin{gather}\label{domire}
\frac{\la^+_i(u)}{\la^+_{i+1}(u)}=q^{\deg P_i}\ts \frac{P_i\big(uq^{-2}\big)}{P_i(u)}=
\frac{\la^-_i(u)}{\la^-_{i+1}(u)},\qquad i=1,\dots,n-1,
\end{gather}
and
\begin{alignat*}{2}
& \frac{\la^+_n(u)}{\la^+_{n+1}(u)}=q^{\frac12\deg P_n}\ts \frac{P_n\big(uq^{-1}\big)}{P_n(u)}=
\frac{\la^-_n(u)}{\la^-_{n+1}(u)}\qquad&&\text{for type $B_n$},
\\
& \frac{\la^+_n(u)}{\la^+_{n+1}(u)}=q^{2\deg P_n}\ts \frac{P_n\big(uq^{-4}\big)}{P_n(u)}=
\frac{\la^-_n(u)}{\la^-_{n+1}(u)}\qquad&&\text{for type $C_n$},
\\
& \frac{\la^+_{n-1}(u)}{\la^+_{n+1}(u)}=q^{\deg P_n}\ts \frac{P_n\big(uq^{-2}\big)}{P_n(u)}=
\frac{\la^-_{n-1}(u)}{\la^-_{n+1}(u)}\qquad&&\text{for type $D_n$},
\end{alignat*}
for some polynomials $P_i(u)$ in $u$, all with constant term $1$, where the first and second
equalities are regarded in $\CC[[u]]$ and $\CC\big[\big[u^{-1}\big]\big]$, respectively.

\item[$2.$]
Every $n$-tuple $(P_1(u),\dots,P_n(u))$,
where each $P_i(u)$ is a polynomial in $u$ with constant term $1$,
arises in this way.

\item[$3.$]
The series $\la^{\pm}_i(u)$ satisfy the relations
\begin{gather}\label{larel}
\la^{\pm}_i\big(u\xi q^{2i}\big)\tss \la^{\pm}_{i'}(u)=\la^{\pm}_{i+1}\big(u\xi q^{2i}\big)\tss \la^{\pm}_{(i+1)'}(u),
\end{gather}
where $i=0,1,\dots,n$ for $N=2n+1$, and $i=0,1,\dots,n-1$ for $N=2n$, and we set
$\la^{\pm}_0(u)=\la^{\pm}_{0'}(u):=1$.
\end{enumerate}
\end{Theorem}

\begin{proof}Using the isomorphism $\U_q(\whg)\cong \U^{R}_q(\whg)$ and the classification results
recalled in Section~\ref{subsec:cr}, we find that any type $\vac$ finite-dimensional irreducible representation $V$ of the algebra~$\U_q(\whg)$
is generated by a vector $\ze$ such that
\begin{alignat}{2}
& e^{\pm}_{i,i+1}(u)\tss\ze=0\qquad&&\text{for}\quad i=1,\dots,n,
\label{eiipoze}\\
& h^{\pm}_{i}(u)\ts\ze=\la^{\pm}_i(u)\ts\ze\qquad&&\text{for}\quad i=1,\dots,n+1,\nonumber
\end{alignat}
for some formal series $\la^{+}_i(u)\in\CC[[u]]$ and $\la^{-}_i(u)\in\CC\big[\big[u^{-1}\big]\big]$,
where for type $D_n$ relation~\eqref{eiipoze} with $i=n$
should be replaced with $e^{\pm}_{n-1,n+1}(u)\tss\ze=0$.
Proposition~\ref{prop:consi} implies that $V$ is a highest weight representation
of $\U^{R}_q(\whg)$. Relations~\eqref{domire} follow from
the results of \cite[Section~12.2]{cp:gq}. They can also be derived by considering
a subalgebra $\U_i\subset \U^{R}_q(\whg)$ associated with the $i$-th simple root of $\g$
such that $\U_i\cong \U_q\big(\wh\sll_2\big)$.
The cyclic span $\U_i \ze$ is a finite-dimensional $\U_q(\wh\sll_2)$-module
which yields the required conditions on the series $\la^{\pm}_i(u)$;
see \cite[Section~3.1]{gm:rt} for more details on this approach going back to~\cite{t:im}.

Part~2 of the theorem follows from the classification results of \cite[Section~12.2]{cp:gq}.
The proof relies on the Hopf algebra structure
on $\U^{R}_q(\whg)$ introduced in
Section~\ref{subsubsec:ha}. It is implied by a~well-known property
of tensor products of representations. Namely,
suppose that~$V$ and~$W$ are finite-dimensional irreducible representations
of $\U^{R}_q(\whg)$ with respective highest vectors $\ze$ and $\eta$,
with the parameter series $\big(\la^{\pm}_1(u),\dots,\la^{\pm}_N(u)\big)$
and $\big(\mu^{\pm}_1(u),\dots,\mu^{\pm}_N(u)\big)$. The coproduct formula~\eqref{Delta}
implies that the irreducible quotient~$X$ of the cyclic span
\begin{gather*}
\U^{R}_q(\whg)(\ze\ot\eta)\subset V\ot W
\end{gather*}
is a highest weight representation with
the parameter series $\big(\la^{\pm}_1(u)\mu^{\pm}_1(u),\dots,\la^{\pm}_N(u)\mu^{\pm}_N(u)\big)$.
Therefore, if $V$ and $W$ are associated with the respective $n$-tuples
of polynomials
\begin{gather*}
(P_1(u),\dots,P_n(u))\Fand(Q_1(u),\dots,Q_n(u)),
\end{gather*}
then
by the formulas of Part~1, the representation $X$
is associated with the $n$-tuple
\begin{gather*}
\big(P_1(u)Q_1(u),\dots,P_n(u)Q_n(u)\big).
\end{gather*}
Hence, to complete the proof of Part~2, one only needs
to produce a finite-dimensional irreducible representation to each
$n$-tuple of the form
\begin{gather*}
\big(1,\dots,1,au+1,1,\dots,1\big),\qquad a\in \CC.
\end{gather*}
The existence of such
{\em fundamental representations} of the quantum affine algebra was established in
\cite{cp:ga}.
Part~3 of the theorem is immediate from Lemma~\ref{lem:gaurell}.
\end{proof}

\begin{Corollary}\label{cor:repe}
All statements of Theorem~{\rm \ref{thm:classi}} hold in the same form
for the algebra $\U^{\ext}_q(\whg)$, except that the value $i=0$ is excluded for the conditions
\eqref{larel}.
\end{Corollary}

\begin{proof}
The arguments used in the proof of the theorem equally apply to the representations of
the algebra $\U^{\ext}_q(\whg)$. Relation~\eqref{larel} with $i=0$
is now replaced by the property that the series~$z^{\pm}(u)$ acts in the highest weight
representation as multiplication by $\la^{\pm}_{1}(u\xi)\tss \la^{\pm}_{1'}(u)$.
\end{proof}

\subsubsection[Representations of quantum affine algebras in type A]{Representations of quantum affine algebras in type $\boldsymbol{A}$}\label{subsubsec:sera}

Corresponding versions of the classification results described in Theorem~\ref{thm:classi}
for the quantum affine algebras in type $A$ are well-known.
They can be derived in the same way as for the types~$B$,~$C$ and~$D$
from the isomorphism between the $R$-matrix and Drinfeld presentations
of the quantum affine algebra $\U_q(\whg)$ for $\g=\sll_n$ constructed in~\cite{df:it}. On the other hand, an independent
proof of the classification theorem
for the quantum affine algebras in the $R$-matrix presentation,
which we state below, was given in~\cite{gm:rt}
following the original approach of~\cite{t:im}.

The quantum affine algebra $\U_q(\wh\sll_n)$ (with the trivial central charge)
is defined by the formulas of Section~\ref{subsec:cr}, where the simple roots
are chosen in the form~\eqref{aroots}.
Finite-dimensional irreducible representations of $\U_q(\wh\sll_n)$ are described
by the results recalled in that section.

Consider the $R$-matrix defined by
\begin{gather*}
R_A(u)=\sum_{i,j=1}^n\tss \big(u\tss q^{\de_{ij}}-q^{-\de_{ij}}\big)\tss e_{ii}\ot e_{jj}
+\big(q-\qin\big)\tss\sum_{i,j=1}^n(u\ts\delta_{i<j} +\delta_{i>j})\ts e_{ij}\ot e_{ji}.
\end{gather*}
The quantum affine algebra $\U^{R}_q\big(\wh\gl_n\big)$
(with the trivial central charge) is generated by
elements~${l}^{\pm}_{ij}[\mp m]$ with $1\leqslant i,j\leqslant n$ and $m\in \ZZ_{+}$
subject to the defining relations described in \eqref{lpmo}--\eqref{rllmp},
where the parameter~$N$ is replaced by~$n$ and the $R$-matrix \eqref{ru}
is replaced by~$R_A(u)$. The Hopf algebra structure on $\U^{R}_q\big(\wh\gl_n\big)$ is
described by the maps \eqref{Delta}, \eqref{S} and~\eqref{cou}.
The algebra $\U^{R}_q\big(\wh\sll_n\big)$ is a Hopf subalgebra
of $\U^{R}_q\big(\wh\gl_n\big)$ which consists of all elements which are stable
with respect to all automorphisms
\begin{gather*}
L^{\pm}(u)\mapsto f^{\pm}(u)\tss L^{\pm}(u),
\end{gather*}
where $f^{\pm}(u)$ are arbitrary series satisfying conditions \eqref{furu}.
By the results of~\cite{df:it} the maps~\eqref{isoma} define a Hopf algebra isomorphism
$\U_q\big(\wh\sll_n\big)\to \U^{R}_q\big(\wh\sll_n\big)$, where the Gaussian generators are
defined by~\eqref{gaussdec}.
The algebra $\U^{R}_q\big(\wh\sll_n\big)$ \big(as well as $\U^{R}_q\big(\wh\gl_n\big)$\big) admits
a family of automorphisms~\eqref{autom}
parameterized by diagonal matrices $\Sc=\diag\ts [\si_1,\dots,\si_{n}]$
such that $\si_i=\pm 1$ for all~$i$.

Highest weight representations of the algebra $\U^{R}_q(\wh\gl_n)$ are defined
in the same way as in Definition~\ref{def:hwr}.
By restrictions, we get highest weight representations of
the subalgebra $\U^{R}_q\big(\wh\sll_n\big)$.
Proposition~\ref{prop:equiv} holds in the same form.
The following theorem is contained in \cite[Theorem~3.6]{gm:rt}.
The proof of Theorem~\ref{thm:classi} applied to the
quantum affine algebras of type $A$ provides
another derivation of that result.

\begin{Theorem}\label{thm:classia}\quad
\begin{enumerate}\itemsep=0pt
\item[$1.$]
Any finite-dimensional irreducible representation of the algebra $\U^{R}_q\big(\wh\gl_n\big)$
{\rm \big(}and $\U^{R}_q\big(\wh\sll_n\big)${\rm \big)}
is a highest weight representation. Up to twisting this representation
with a suitable automorphism \eqref{autom}, its parameters satisfy the relations
\begin{gather*}
\frac{\la^+_i(u)}{\la^+_{i+1}(u)}=q^{\deg P_i}\ts \frac{P_i\big(uq^{-2}\big)}{P_i(u)}=
\frac{\la^-_i(u)}{\la^-_{i+1}(u)},\qquad i=1,\dots,n-1,
\end{gather*}
for some polynomials $P_i(u)$ in $u$, all with constant term $1$,
where the equalities are understood for the expansions of the rational function in $u$
as a series in $u$ and $u^{-1}$, respectively.

\item[$2.$]
Every $n$-tuple of polynomials $(P_1(u),\dots,P_{n-1}(u))$ in $u$,
where each $P_i(u)$ has constant term~$1$,
arises in this way.
\end{enumerate}
\end{Theorem}

\subsubsection{Outlook}

The isomorphisms of \cite{jlm:ib-c,jlm:ib-bd}
do not rely on any Poincar\'{e}--Birkhoff--Witt-type theorem for
the algebra $\U^{R}_q(\whg)$. Therefore, such a theorem is implied by
the results of Beck~\cite{b:bg}. In particular,
by Proposition~\ref{prop:consi}, the image of the basis
of the subalgebra $\U_q(\whg)^+$ under the isomorphism \mbox{$\U_q(\whg)\to \U^R_q(\whg)$}
is a basis of the subalgebra $\U^R_q(\whg)^+$ generated by the coefficients
of all series~$e^{\pm}_{ij}(u)$. However, a precise expression of the basis elements
in terms of these coefficients is unknown.

Note also that the version of the Poincar\'{e}--Birkhoff--Witt theorem for
the algebra $\U_q\big(\wh\gl_N\big)$ given in \cite[Corollary~2.13]{gm:rt}
does not immediately extend to types~$B$,~$C$ and $D$, because of more complicated
defining relations on the generators $l^{\pm}_{ij}[m]$.
It would be interesting to find an alternative form of
the quadratic relations for the generator
series to lead to such a version.

\appendix

\section{Modified isomorphism for Yangians}\label{asec:isy}

Here we give a
modified version of the isomorphism produced in
\cite[Main Theorem]{jlm:ib} which can be used to establish a different
correspondence between the parameters of representations in the two realizations
of the Yangians; cf.\ Section~\ref{subsubsec:isy}.
An analogous isomorphism was used for type $A$ in \cite[Section~3.1]{m:yc}; see also
\cite{bk:pp} for an isomorphism with an opposite presentation of the Yangian $\Y(\sll_N)$.
The new version is based on the alternative Gauss decomposition of the mat\-rix~$T(u)$, defined by
\begin{gather}\label{mgd}
T(u)=\overline E(u)\ts \overline H(u)\ts \overline F(u),
\end{gather}
where $\overline E(u)$, $\overline H(u)$ and $\overline F(u)$
are uniquely determined matrices of the form
\begin{gather*}
\overline E(u)=\begin{bmatrix}
\ts1&\bar e_{12}(u)&\dots&\bar e_{1N}(u)\ts\\
\ts0&1&\dots&\bar e_{2N}(u)\\
\vdots&\vdots&\ddots&\vdots\\
0&0&\dots&1
\end{bmatrix},
\qquad
\overline F(u)=\begin{bmatrix}
1&0&\dots&0\ts\\
\bar f_{21}(u)&1&\dots&0\\
\vdots&\vdots&\ddots&\vdots\\
\bar f_{N1}(u)&\bar f_{N2}(u)&\dots&1
\end{bmatrix},
\end{gather*}
and $\overline H(u)=\diag\ts \big[\bar h_1(u),\dots,\bar h_N(u) \big]$. Define the series
with coefficients in $\Y(\g)$ by
\begin{gather*}
\kappa^{}_i(u)=\bar h_i (u+(i-1)/2 )\ts \bar h_{i+1} (u+(i-1)/2 )^{-1}
\end{gather*}
for $i=1,\dots,n-1$, and
\begin{gather*}
\kappa^{}_n(u)=\begin{cases}
\bar h_n (u+(n-1)/2 )\ts \bar h_{n+1} (u+(n-1)/2 )^{-1}
 &\text{for}\ \oa_{2n+1},\\
\tss \bar h_n (u+n/2 )\ts \bar h_{n+1} (u+n/2 )^{-1}
 &\text{for}\ \spa_{2n},\\
\bar h_{n-1} (u+(n-2)/2 )\ts \bar h_{n+1} (u+(n-2)/2 )^{-1}
 &\text{for}\ \oa_{2n}.
\end{cases}
\end{gather*}
Furthermore, set
\begin{gather*}
\xi_i^+(u)=\bar e_{i\ts i+1}\big(u+(i-1)/2\big),\qquad
\xi_i^-(u)=\bar f_{i+1\ts i}\big(u+(i-1)/2\big)
\end{gather*}
for $i=1,\dots,n-1$,
\begin{gather*}
\xi_n^+(u)=\begin{cases}
\bar e_{n\ts n+1} (u+(n-1)/2 \big)
 &\text{for}\ \oa_{2n+1},\\
\bar e_{n\ts n+1}\big(u+n/2\big)
 &\text{for}\ \spa_{2n},\\
\bar e_{n-1\ts n+1}\big(u+(n-2)/2\big)
 &\text{for}\ \oa_{2n}
\end{cases}
\end{gather*}
and
\begin{gather*}
\xi_n^-(u)=\begin{cases}
\bar f_{n+1\ts n} (u+(n-1)/2 )
 &\text{for}\ \oa_{2n+1},\\
\frac{1}{2}\bar f_{n+1\ts n} (u+n/2 )
 &\text{for}\ \spa_{2n},\\
\bar f_{n+1\ts n-1} (u+(n-2)/2 )
 &\text{for}\ \oa_{2n}.
\end{cases}
\end{gather*}
Introduce elements of $\Y(\g)$ by the respective expansions
into power series in $u^{-1}$,
\begin{gather*}
\kappa^{}_i(u)=1+\sum_{r=0}^{\infty}\kappa^{}_{i\tss r}\ts u^{-r-1}\Fand
\xi_i^{\pm}(u)=\sum_{r=0}^{\infty}\xi_{i\tss r}^{\pm}\ts u^{-r-1}
\end{gather*}
for $i=1,\dots,n$.

\begin{Theorem}\label{thm:misom}
The mapping which sends the generators
$\kappa^{}_{i\tss r}$ and $\xi_{i\tss r}^{\pm}$ of $\Y^{D}(\g)$ to the elements
of $\Y(\g)$ with the same names defines an isomorphism
$\Y^{D}(\g)\cong \Y(\g)$.
\end{Theorem}

\begin{proof}
We will derive this result from \cite[Main Theorem]{jlm:ib}
as recalled in Section~\ref{subsubsec:isy}, by taking the composition
of the isomorphism
$\Y^{D}(\g)\to\Y(\g)$ constructed therein, with certain automorphisms
of the algebra $\Y(\g)$.
We have the following quasideterminant formulas for
the entries of the matrices $\overline H(u)$, $\overline E(u)$ and $\overline F(u)$ \cite{gr:tn}:
\begin{gather}\label{myhmqua}
\bar h_i(u)=\begin{vmatrix} \boxed{t_{i\tss i}(u)}&t_{i\ts i+1}(u)&\dots&t_{i\tss N}(u)\\
 t_{i+1\ts i}(u)&t_{i+1\ts i+1}(u)&\dots&t_{i+1\tss N}(u)\\
 \vdots&\vdots&\ddots&\vdots\\
 t_{N\tss i}(u)&t_{N\ts i+1}(u)&\dots&t_{N\tss N}(u)
 \end{vmatrix},\qquad i=1,\dots,N,
\end{gather}
whereas
\begin{gather}\label{myeijmlqua}
\bar e_{ij}(u)=\begin{vmatrix} \boxed{t_{i\tss j}(u)}&t_{i\ts j+1}(u)&\dots&t_{i\tss N}(u)\\
 t_{j+1\ts j}(u)&t_{j+1\ts j+1}(u)&\dots&t_{j+1\tss N}(u)\\
 \vdots&\vdots&\ddots&\vdots\\
 t_{N\tss j}(u)&t_{N\ts j+1}(u)&\dots&t_{N\tss N}(u)
 \end{vmatrix} \ts h_j(u)^{-1}
\end{gather}
and
\begin{gather}\label{myfijlmqua}
\bar f_{ji}(u)=h_j(u)^{-1}\ts\begin{vmatrix} \boxed{t_{j\tss i}(u)}&t_{j\ts j+1}(u)&\dots&t_{j\tss N}(u)\\
 t_{j+1\ts i}(u)&t_{j+1\ts j+1}(u)&\dots&t_{j+1\tss N}(u)\\
 \vdots&\vdots&\ddots&\vdots\\
 t_{N\tss i}(u)&t_{N\ts j+1}(u)&\dots&t_{N\tss N}(u)
 \end{vmatrix}
\end{gather}
for $1\leqslant i<j\leqslant N$. The Gaussian generators $h_i(u)$, $e_{ij}(u)$ and $f_{ji}(u)$
are defined by using the decomposition \eqref{gd} dual to \eqref{mgd},
and are given by
the respective dual quasideterminant formulas; see \cite[Section~4]{jlm:ib}.
On the other hand, the mapping
\begin{gather}\label{mausi}
\vs\colon \ t_{ij}(u)\mapsto t_{i'j'}(u),\qquad 1\leqslant i,j\leqslant N,
\end{gather}
defines an involutive automorphism of the algebra $\Y(\g)$.
Since quasideterminants are unchanged under permutations of rows or columns,
we find that the images of the Gaussian generators are given by
\begin{gather*}
\vs\colon \ h_i(u)\mapsto \bar h_{i'}(u),\qquad e_{ij}(u)\mapsto \bar f_{i'j'}(u),
\qquad f_{ji}(u)\mapsto \bar e_{j'i'}(u).
\end{gather*}
Furthermore, the unitary condition \eqref{unitaryint} implies symmetry
relations for the Gaussian generators which were described in
\cite[Section~5.3]{jlm:ib}.\footnote{The formulas in \cite[Proposition~5.7]{jlm:ib} concerning
the Lie algebra $\oa_{2n+1}$
should be corrected as follows: $e_{n+1\ts n+2}(u)=-e_n(u-1/2)$ and $f_{n+2\ts n+1}(u)=-f_n(u-1/2)$.}
They are given in \eqref{yhiipone}, and
for $i=1,\dots,n-1$ we have
\begin{gather*}
e_{(i+1)'\ts i'}(u)=-e_{i\ts i+1}(u+\ka-i)\Fand
f_{i'\ts (i+1)'}(u)=-f_{i+1\ts i}(u+\ka-i),
\end{gather*}
with some additional type-specific relations.
These relations allow us to express the images of the generators of $\Y^{D}(\g)$
under the composition of the isomorphism $\Y^{D}(\g)\to\Y(\g)$ provided by
\cite[Main Theorem]{jlm:ib}
with the automorphism \eqref{mausi}, in terms of the Gaussian generators
\eqref{myhmqua}--\eqref{myfijlmqua}.
The formulas given in the statement of the theorem are obtained
by taking further compositions with the shift automorphism $T(u)\mapsto T(u+\ka-1)$
and with the automorphism which multiplies all generators
$\xi_{i\tss r}^{\pm}$ by $-1$, while leaving all $\kappa^{}_{i\ts r}$ unchanged.
\end{proof}

\subsection*{Acknowledgements}

We acknowledge the support of the Australian Research Council, grant DP180101825.

\pdfbookmark[1]{References}{ref}
\LastPageEnding


\begin{thebibliography}{99}
\footnotesize\itemsep=0pt

\bibitem{amr:rp}
Arnaudon D., Molev A., Ragoucy E., On the {$R$}-matrix realization of
 {Y}angians and their representations, \href{https://doi.org/10.1007/s00023-006-0281-9}{\textit{Ann. Henri Poincar\'e}}
 \textbf{7} (2006), 1269--1325, \href{https://arxiv.org/abs/math.QA/0511481}{arXiv:math.QA/0511481}.

\bibitem{b:ts}
Bazhanov V.V., Trigonometric solutions of triangle equations and classical
 {L}ie algebras, \href{https://doi.org/10.1016/0370-2693(85)90259-X}{\textit{Phys. Lett.~B}} \textbf{159} (1985), 321--324.

\bibitem{b:iq}
Bazhanov V.V., Integrable quantum systems and classical {L}ie algebras,
 \href{https://doi.org/10.1007/BF01221256}{\textit{Comm. Math. Phys.}} \textbf{113} (1987), 471--503.

\bibitem{b:bg}
Beck J., Braid group action and quantum affine algebras, \href{https://doi.org/10.1007/BF02099423}{\textit{Comm. Math.
 Phys.}} \textbf{165} (1994), 555--568, \href{https://arxiv.org/abs/hep-th/9404165}{arXiv:hep-th/9404165}.

\bibitem{bk:pp}
Brundan J., Kleshchev A., Parabolic presentations of the {Y}angian
 {$Y({\mathfrak{gl}}_n)$}, \href{https://doi.org/10.1007/s00220-004-1249-6}{\textit{Comm. Math. Phys.}} \textbf{254} (2005),
 191--220, \href{https://arxiv.org/abs/math.QA/0407011}{arXiv:math.QA/0407011}.

\bibitem{cp:gq}
Chari V., Pressley A., A guide to quantum groups, Cambridge University Press,
 Cambridge, 1994.

\bibitem{cp:ga}
Chari V., Pressley A., Quantum affine algebras and their representations, in
 Representations of Groups ({B}anff, {AB}, 1994), \textit{CMS Conf. Proc.},
 Vol.~16, \href{https://doi.org/10.1007/bf00750760}{Amer. Math. Soc.}, Providence, RI, 1995, 59--78,
 \href{https://arxiv.org/abs/hep-th/9411145}{arXiv:hep-th/9411145}.

\bibitem{df:it}
Ding J.T., Frenkel I.B., Isomorphism of two realizations of quantum affine
 algebra {$U_q(\mathfrak{gl}(n))$}, \href{https://doi.org/10.1007/BF02098484}{\textit{Comm. Math. Phys.}} \textbf{156}
 (1993), 277--300.

\bibitem{d:ha}
Drinfeld V.G., Hopf algebras and the quantum {Y}ang--{B}axter equation,
 \textit{Soviet Math. Dokl.} \textbf{32} (1985), 254--258.

\bibitem{d:qg}
Drinfeld V.G., Quantum groups, in Proceedings of the {I}nternational {C}ongress
 of {M}athematicians, {V}ols. 1,~2 ({B}erkeley, {C}alif., 1986), Amer. Math.
 Soc., Providence, RI, 1987, 798--820.

\bibitem{d:nr}
Drinfeld V.G., A new realization of {Y}angians and of quantum affine algebras,
 \textit{Soviet Math. Dokl.} \textbf{36} (1988), 212--216.

\bibitem{fri:qa}
Frenkel I.B., Reshetikhin N.Yu., Quantum affine algebras and holonomic
 difference equations, \href{https://doi.org/10.1007/BF02099206}{\textit{Comm. Math. Phys.}} \textbf{146} (1992), 1--60.

\bibitem{gr:tn}
Gel'fand I.M., Retakh V.S., A theory of noncommutative determinants and
 characteristic functions of graphs, \href{https://doi.org/10.1007/BF01075044}{\textit{Funct. Anal. Appl.}} \textbf{26}
 (1992), 231--246.

\bibitem{gm:rt}
Gow L., Molev A., Representations of twisted {$q$}-{Y}angians, \href{https://doi.org/10.1007/s00029-010-0030-2}{\textit{Selecta
 Math. (N.S.)}} \textbf{16} (2010), 439--499, \href{https://arxiv.org/abs/0909.4905}{arXiv:0909.4905}.

\bibitem{grw:eb}
Guay N., Regelskis V., Wendlandt C., Equivalences between three presentations
 of orthogonal and symplectic {Y}angians, \href{https://doi.org/10.1007/s11005-018-1108-6}{\textit{Lett. Math. Phys.}}
 \textbf{109} (2019), 327--379, \href{https://arxiv.org/abs/1706.05176}{arXiv:1706.05176}.

\bibitem{h:rq}
Hernandez D., Representations of quantum affinizations and fusion product,
 \href{https://doi.org/10.1007/s00031-005-1005-9}{\textit{Transform. Groups}} \textbf{10} (2005), 163--200,
 \href{https://arxiv.org/abs/math.QA/0312336}{arXiv:math.QA/0312336}.

\bibitem{j:qd}
Jimbo M., A {$q$}-difference analogue of {$U({\mathfrak g})$} and the
 {Y}ang--{B}axter equation, \href{https://doi.org/10.1007/BF00704588}{\textit{Lett. Math. Phys.}} \textbf{10} (1985),
 63--69.

\bibitem{j:qr}
Jimbo M., Quantum {$R$} matrix for the generalized {T}oda system, \href{https://doi.org/10.1007/BF01221646}{\textit{Comm.
 Math. Phys.}} \textbf{102} (1986), 537--547.

\bibitem{jlm:ib}
Jing N., Liu M., Molev A., Isomorphism between the {$R$}-matrix and {D}rinfeld
 presentations of {Y}angian in types {$B$}, {$C$} and {$D$}, \href{https://doi.org/10.1007/s00220-018-3185-x}{\textit{Comm.
 Math. Phys.}} \textbf{361} (2018), 827--872, \href{https://arxiv.org/abs/1705.08155}{arXiv:1705.08155}.

\bibitem{jlm:ib-c}
Jing N., Liu M., Molev A., Isomorphism between the {$R$}-matrix and {D}rinfeld
 presentations of quantum affine algebra: type~{$C$}, \href{https://doi.org/10.1063/1.5133854}{\textit{J.~Math. Phys.}}
 \textbf{61} (2020), 031701, 41~pages, \href{https://arxiv.org/abs/1903.00204}{arXiv:1903.00204}.

\bibitem{jlm:ib-bd}
Jing N., Liu M., Molev A., Isomorphism between the {$R$}-matrix and {D}rinfeld
 presentations of quantum affine algebra: types~{$B$} and~{$D$},
 \href{https://doi.org/10.3842/SIGMA.2020.043}{\textit{SIGMA}} \textbf{16} (2020), 043, 49~pages, \href{https://arxiv.org/abs/1911.03496}{arXiv:1911.03496}.

\bibitem{kpt:ob}
Khoroshkin S., Pakuliak S., Tarasov V., Off-shell {B}ethe vectors and
 {D}rinfeld currents, \href{https://doi.org/10.1016/j.geomphys.2007.02.005}{\textit{J.~Geom. Phys.}} \textbf{57} (2007), 1713--1732,
 \href{https://arxiv.org/abs/math.QA/0610517}{arXiv:math.QA/0610517}.

\bibitem{ks:qs}
Kulish P.P., Sklyanin E.K., Quantum spectral transform method. {R}ecent
 developments, in Integrable Quantum Field Theories ({T}v\"arminne, 1981),
 \textit{Lecture Notes in Phys.}, Vol.~151, \href{https://doi.org/10.1007/3-540-11190-5_8}{Springer}, Berlin~-- New York,
 1982, 61--119.

\bibitem{m:yc}
Molev A., Yangians and classical {L}ie algebras, \textit{Mathematical Surveys
 and Monographs}, Vol.~143, \href{https://doi.org/10.1090/surv/143}{Amer. Math. Soc.}, Providence, RI, 2007.

\bibitem{rs:ce}
Reshetikhin N.Yu., Semenov-Tian-Shansky M.A., Central extensions of quantum
 current groups, \href{https://doi.org/10.1007/BF01045884}{\textit{Lett. Math. Phys.}} \textbf{19} (1990), 133--142.

\bibitem{rtf:ql}
Reshetikhin N.Yu., Takhtadzhyan L.A., Faddeev L.D., Quantization of {L}ie groups
 and {L}ie algebras, \textit{Leningrad Math.~J.} \textbf{1} (1990), 193--225.

\bibitem{t:sq}
Tarasov V.O., Structure of quantum $L$-operators for the $R$-matrix of the
 $XXZ$-model, \href{https://doi.org/10.1007/BF01029107}{\textit{Theoret. and Math. Phys.}} \textbf{61} (1984),
 1065--1072.

\bibitem{t:im}
Tarasov V.O., Irreducible monodromy matrices for the {$R$}-matrix of the
 {$XXZ$}-model and local lattice quantum {H}amiltonians, \href{https://doi.org/10.1007/BF01017900}{\textit{Theoret. and
 Math. Phys.}} \textbf{63} (1985), 440--454.

\bibitem{w:rm}
Wendlandt C., The {$R$}-matrix presentation for the {Y}angian of a simple {L}ie
 algebra, \href{https://doi.org/10.1007/s00220-018-3227-4}{\textit{Comm. Math. Phys.}} \textbf{363} (2018), 289--332,
 \href{https://arxiv.org/abs/1709.08162}{arXiv:1709.08162}.

\bibitem{zz:rf}
Zamolodchikov A.B., Zamolodchikov A.B., Factorized {$S$}-matrices in two
 dimensions as the exact solutions of certain relativistic quantum field
 theory models, \href{https://doi.org/10.1016/0003-4916(79)90391-9}{\textit{Ann. Physics}} \textbf{120} (1979), 253--291.

\end{thebibliography}
\end{document}